\documentclass[a4paper,11pt,reqno]{amsart}

\usepackage[T1]{fontenc}
\usepackage[utf8]{inputenc}
\usepackage{lmodern}
\usepackage[english]{babel}

\usepackage[%
a4paper,
left=2.5cm, 
right=2.5cm, 
top=3.5cm, 
bottom=3.5cm, 
heightrounded, 
bindingoffset=0mm 
]{geometry}

\usepackage{amsmath}
\usepackage{amssymb}
\usepackage{mathtools}
\usepackage{mathrsfs}
\usepackage[dvipsnames]{xcolor}
\usepackage{graphicx}
\usepackage{enumitem}
\usepackage{hyperref} 
\usepackage[normalem]{ulem}

\usepackage[noabbrev,capitalize]{cleveref}

\usepackage{bookmark}
\usepackage[initials]{amsrefs}
\renewcommand{\MR}[1]{} 

\numberwithin{equation}{section}

\theoremstyle{definition}
\newtheorem{definition}{Definition}[section]

\theoremstyle{plain}
\newtheorem{proposition}[definition]{Proposition}
\newtheorem{corollary}[definition]{Corollary}
\newtheorem{theorem}[definition]{Theorem}
\newtheorem{lemma}[definition]{Lemma}
\newtheorem{assumption}[definition]{Assumption}

\theoremstyle{remark}
\newtheorem{remark}[definition]{Remark}

\DeclarePairedDelimiter{\abs}{|}{|}
\DeclarePairedDelimiter{\norm}{\|}{\|}
\DeclarePairedDelimiter{\set}{\{}{\}}

\DeclareMathOperator{\divergence}{div}
\DeclareMathOperator{\di}{d\!}
\DeclareMathOperator{\esssup}{ess\,sup}

\newcommand{\e}{\varepsilon}
\newcommand{\R}{\mathbb{R}}
\newcommand{\Lip}{\textrm{Lip}}
\newcommand{\ul}{\mathrm{ul}}
\newcommand{\loc}{\mathrm{loc}}

\renewcommand{\phi}{\varphi}

\newcommand{\XX}{\mathcal{X}}
\newcommand{\VV}{\mathcal{V}}

\newcommand{\wass}{\mathsf{W}}
\newcommand{\plan}{\mathsf{Plan}}
\newcommand{\proj}{\mathsf{p}}
\newcommand{\prob}{{\mathscr P}}

\newcommand{\class}{\mathcal A}

\makeatletter
\newcommand*{\mint}[1]{%
  \mint@l{#1}{}%
}
\newcommand*{\mint@l}[2]{%
  \@ifnextchar\limits{%
    \mint@l{#1}%
  }{%
    \@ifnextchar\nolimits{%
      \mint@l{#1}%
    }{%
      \@ifnextchar\displaylimits{%
        \mint@l{#1}%
      }{%
        \mint@s{#2}{#1}%
      }%
    }%
  }%
}
\newcommand*{\mint@s}[2]{%
  \@ifnextchar_{%
    \mint@sub{#1}{#2}%
  }{%
    \@ifnextchar^{%
      \mint@sup{#1}{#2}%
    }{%
      \mint@{#1}{#2}{}{}%
    }%
  }%
}
\def\mint@sub#1#2_#3{%
  \@ifnextchar^{%
    \mint@sub@sup{#1}{#2}{#3}%
  }{%
    \mint@{#1}{#2}{#3}{}%
  }%
}
\def\mint@sup#1#2^#3{%
  \@ifnextchar_{%
    \mint@sup@sub{#1}{#2}{#3}%
  }{%
    \mint@{#1}{#2}{}{#3}%
  }%
}
\def\mint@sub@sup#1#2#3^#4{%
  \mint@{#1}{#2}{#3}{#4}%
}
\def\mint@sup@sub#1#2#3_#4{%
  \mint@{#1}{#2}{#4}{#3}%
}
\newcommand*{\mint@}[4]{%
  \mathop{}%
  \mkern-\thinmuskip
  \mathchoice{%
    \mint@@{#1}{#2}{#3}{#4}%
        \displaystyle\textstyle\scriptstyle
  }{%
    \mint@@{#1}{#2}{#3}{#4}%
        \textstyle\scriptstyle\scriptstyle
  }{%
    \mint@@{#1}{#2}{#3}{#4}%
        \scriptstyle\scriptscriptstyle\scriptscriptstyle
  }{%
    \mint@@{#1}{#2}{#3}{#4}%
        \scriptscriptstyle\scriptscriptstyle\scriptscriptstyle
  }%
  \mkern-\thinmuskip
  \int#1%
  \ifx\\#3\\\else_{#3}\fi
  \ifx\\#4\\\else^{#4}\fi  
}
\newcommand*{\mint@@}[7]{%
  \begingroup
    \sbox0{$#5\int\m@th$}%
    \sbox2{$#5\int_{}\m@th$}%
    \dimen2=\wd0 %
    \let\mint@limits=#1\relax
    \ifx\mint@limits\relax
      \sbox4{$#5\int_{\kern1sp}^{\kern1sp}\m@th$}%
      \ifdim\wd4>\wd2 %
        \let\mint@limits=\nolimits
      \else
        \let\mint@limits=\limits
      \fi
    \fi
    \ifx\mint@limits\displaylimits
      \ifx#5\displaystyle
        \let\mint@limits=\limits
      \fi
    \fi
    \ifx\mint@limits\limits
      \sbox0{$#7#3\m@th$}%
      \sbox2{$#7#4\m@th$}%
      \ifdim\wd0>\dimen2 %
        \dimen2=\wd0 %
      \fi
      \ifdim\wd2>\dimen2 %
        \dimen2=\wd2 %
      \fi
    \fi
    \rlap{%
      $#5%
        \vcenter{%
          \hbox to\dimen2{%
            \hss
            $#6{#2}\m@th$%
            \hss
          }%
        }%
      $%
    }%
  \endgroup
}

\allowdisplaybreaks
\begin{document}

\title[Existence and stability of weak solutions of the Vlasov--Poisson system]{Existence and Stability of weak solutions of the Vlasov--Poisson system in localized Yudovich spaces}

\author[G.~Crippa]{Gianluca Crippa}
\address[G.~Crippa]{Department Mathematik und Informatik, Universität Basel, Spiegelgasse~1, 4051 Basel, Switzerland}
\email{gianluca.crippa@unibas.ch}

\author[M.~Inversi]{Marco Inversi}
\address[M.~Inversi]{Department Mathematik und Informatik, Universität Basel, Spiegelgasse~1, 4051 Basel, Switzerland}
\email{marco.inversi@unibas.ch}

\author[C.~Saffirio]{Chiara Saffirio}
\address[C.~Saffirio]{Department Mathematik und Informatik, Universität Basel, Spiegelgasse~1, 4051 Basel, Switzerland}
\email{chiara.saffirio@unibas.ch}

\author[G.~Stefani]{Giorgio Stefani}
\address[G.~Stefani]{Scuola Internazionale Superiore di Studi Avanzati (SISSA), via Bonomea~265, 34136 Trieste (TS), Italy}
\email{giorgio.stefani.math@gmail.com {\normalfont or} gstefani@sissa.it}

\date{\today}



\thanks{\textit{Acknowledgments}. The first and second-named authors have been partially funded by the SNF grant FLUTURA: Fluids, Turbulence, Advection No.~212573. The third-named author acknowledges the NCCR SwissMAP and the support of the SNSF through the Eccellenza project PCEFP2\_181153. The fourth-named author is member of the Istituto Nazionale di Alta Matematica (INdAM), Gruppo Nazionale per l'Analisi Matematica, la Probabilità e le loro Applicazioni (GNAMPA), is partially supported by the INdAM--GNAMPA 2022 Project \textit{Analisi geometrica in strutture subriemanniane}, codice CUP\_E55\-F22\-000\-270\-001 and by the INdAM--GNAMPA 2023 Project \textit{Problemi variazionali per funzionali e operatori non-locali}, codice CUP\_E53\-C22\-001\-930\-001, and has received funding from the European Research Council (ERC) under the European Union’s Horizon 2020 research and innovation program (grant agreement No.~945655). }

\begin{abstract} 
We consider the Vlasov--Poisson system both in the repulsive (electrostatic potential) and in the attractive (gravitational potential) cases. 
In our first main theorem, we prove the uniqueness and the quantitative stability of Lagrangian solutions $f=f(t,x,v)$ whose associated spatial density $\rho_f=\rho_f(t,x)$ is potentially unbounded but belongs to suitable uniformly-localized Yudovich spaces. This requirement imposes a condition of slow growth on the function $p \mapsto \|\rho_f(t,\cdot)\|_{L^p}$ uniformly in time. Previous works by Loeper, Miot and Holding--Miot have addressed the cases of bounded spatial density, i.e., $\|\rho_f(t,\cdot)\|_{L^p} \lesssim 1$, and spatial density such that $\|\rho_f(t,\cdot)\|_{L^p} \sim p^{1/\alpha}$ for $\alpha\in[1,+\infty)$. 
Our approach is Lagrangian and relies on an explicit estimate of the modulus of continuity of the electric field and on a second-order Osgood lemma. It also allows for iterated-logarithmic perturbations of the linear growth condition.
In our second main theorem, we complement the aforementioned result by constructing solutions whose spatial density sharply satisfies such iterated-logarithmic growth. 
Our approach relies on real-variable techniques and extends the strategy developed for the Euler equations by the first and fourth-named authors. It also allows for the treatment of more general equations that share the same structure as the Vlasov--Poisson system. Notably, the uniqueness result and the stability estimates hold for both the classical and the relativistic Vlasov--Poisson systems.
\end{abstract}

\maketitle

\section{Introduction}

\subsection{Framework}

For some fixed $T\in(0,+\infty)$, we consider the \emph{Vlasov--Poisson system} 
\begin{equation} \label{eq:vp}
\left\{\!\!
\begin{array}{ll}
\partial_t f + v\cdot \nabla_x f + E_f \cdot \nabla_v f = 0 
& \text{in}\ (0,T) \times \R^{2d},
\\[2ex] 
E_f(t,x) = \kappa\displaystyle\int_{\R^d} K(x-y) \, \rho_f(t,y) \di y & \text{in}\ (0,T) \times \R^d,
\\[2ex] 
 \rho_f(t,x) = \displaystyle\int_{\R^d} f(t,x,v) \di v & \text{in}\ (0,T) \times \R^d, 
\\[3ex] 
f(0,x,v) = f_0(x,v) & \text{in } \R^{2d},
\end{array}
\right.
\end{equation}
where 
$f_0\in L^1(\R^{2d})$ is the initial datum, 
$f\in L^\infty([0,T];L^1(\R^{2d}))$ is the unknown,
$\rho_f\in L^\infty([0,T];L^1(\R^d))$ is the spatial density associated with $f$, 
$\kappa\in\{-1,+1\}$ 
and  
$K\colon \R^d \to \R^d$ is the \emph{Riesz kernel}, given by
\begin{equation} \label{eq:kernel}
K(z) = \frac{x}{\abs{x}^d},
\quad
x\in\R^d\setminus\set*{0}.
\end{equation} 
In particular, the vector field $E_f\in L^\infty([0,T];L^1_\loc(\R^d;\R^d))$ is well defined.
For $d=3$, the Vlasov--Poisson system~\eqref{eq:vp} describes the time evolution of the density $f$ of plasma consisting of charged particles with long-range interaction, e.g., a repulsive Coulomb potential for $\kappa=1$ or an attracting gravitational potential for $\kappa=-1$.

The Vlasov--Poisson system \eqref{eq:vp} has been extensively investigated. 
Existence and uniqueness of classical solutions of the system~\eqref{eq:vp} under some regularity assumptions on the initial data go back to Iordanski~\cite{I61} for $d=1$ and to Okabe--Ukai~\cite{UO78} for $d=2$.
In any dimension, global existence of weak solutions with finite energy 
\begin{equation*}
\sup_{t\in[0,T]}
\int_{\R^{2d}}|v|^2\,f(t,x,v)\di x\di v
+
\frac\kappa2
\int_{\R^d}|E_f(t,x)|^2\di x
<+\infty
\end{equation*}
is due to Arsen'ev~\cite{A75}.
For $d=3$, global existence and uniqueness have been addressed by Bardos--Degond~\cite{BD85} for classical solutions with small initial data, and then  by Pfaffelmoser~\cite{P92} and Lions--Perthame~\cite{LP91} using different methods. 
The main idea of~\cite{P92} is to exploit \emph{Lagrangian} techniques to prove global existence and uniqueness of classical solutions with compactly supported initial data. 
The approach of~\cite{LP91}, instead, relies on an \emph{Eulerian} point of view, yielding existence of global weak solutions with finite velocity moments.
More precisely, for $d=3$, if $f_0\in L^1(\R^d)\cap L^\infty(\R^d)$ is such that 
\begin{equation}
\label{eq:m-moment-vel}
\int_{\R^{2d}}|v|^m
f_0(x,v)\di x\di v<+\infty
\quad
\text{for some}\ m>3, 
\end{equation}
then there exists a corresponding weak solution $f\in L^\infty([0,+\infty);L^1(\R^{2d}))$ such that 
\begin{equation*}
\sup_{t\in[0,T]}
\int_{\R^{2d}}|v|^m 
f(t,x,v)\di x\di v<+\infty
\quad
\text{for any}\ T>0.
\end{equation*}
For further developments concerning the propagation of moments and global existence of weak solutions of the Vlasov--Poisson system~\eqref{eq:vp}, we refer the reader to~\cites{GJP00,S09,P12,P14,DMS15,CLS18}.

Sufficient conditions for uniqueness of weak solutions of the Vlasov--Poisson system~\eqref{eq:vp} have been first obtained in~\cite{LP91}, provided that~\eqref{eq:m-moment-vel} holds with $m>6$ and a technical assumption on the support of the initial data is satisfied. 
A simpler criterion has been then proposed by Robert~\cite{R97} for compactly supported weak solutions, and later extended by Loeper~\cite{L06} to measure-valued solutions~$f$ with spatial density such that 
\begin{equation}\label{eq:Loeper-growth}
\rho_f\in L^\infty([0,T];L^\infty(\R^d)).
\end{equation}
Recently, Miot~\cite{M16} generalized the uniqueness criterion of~\cite{LP91} to measure-valued solutions~$f$ with spatial density such that, for some $T>0$, 
\begin{equation}\label{eq:Miot-growth}
\sup_{t\in [0,T]}\sup_{p\geq 1} \dfrac{\|\rho_f(t,\cdot)\|_{L^p}}{p}<+\infty.
\end{equation}
The uniqueness condition~\eqref{eq:Miot-growth} is satisfied by some non-trivial weak solutions with  initial data having unbounded macroscopic density, see~\cite{M16}*{Ths.~1.2 and~1.3}.    
Later, Holding--Miot~\cite{HM18} provided a uniqueness criterion  interpolating between the conditions~\eqref{eq:Loeper-growth} and~\eqref{eq:Miot-growth} by considering measure-valued solutions~$f$ with spatial density such that, for some $T>0$ and $\alpha\in[1,+\infty)$,
\begin{equation}
\label{eq:Holding-Miot-growth}
\sup_{t\in [0,T]}\sup_{p\geq \alpha} \dfrac{\|\rho_f(t,\cdot)\|_{L^p}}{p^{1/\alpha}}<+\infty.
\end{equation}
The case $\alpha=1$ corresponds to~\eqref{eq:Miot-growth}, while the limiting case $\alpha=+\infty$ corresponds to~\eqref{eq:Loeper-growth}.
Condition~\eqref{eq:Holding-Miot-growth} implies that~$\rho_f$ belongs to an \emph{exponential Orlicz space}, see~\cite{HM18}*{Sec.~1.1.1}.
Conditions~\eqref{eq:Miot-growth} and~\eqref{eq:Holding-Miot-growth} allow to consider initial data with compact support in velocity as well as \emph{Maxwell--Boltzmann distributions} with exponential decay as $|v|\to+\infty$, see the comments below~\cite{M16}*{Th.~1.2} and~\cite{HM18}*{Prop.~1.14}.

\subsection{Yudovich spaces and modulus of continuity}

The main aim of the present paper is to establish existence and stability properties of weak solutions of the Vlasov--Poisson system~\eqref{eq:vp}, extending the results obtained in~\cites{L06,M16,HM18} to measure-valued solutions with spatial density belonging to \emph{uniformly-localized Yudovich spaces}.

We consider solutions~$f$ of the system~\eqref{eq:vp} whose spatial density~$\rho_f$ satisfies 
\begin{equation}
\label{eq:Yudo-growth}
\sup_{t\in [0,T]}\sup_{p\geq 1} \dfrac{\|\rho_f(t,\cdot)\|_{L^p}}{\Theta(p)}<+\infty
\end{equation}
for some fixed increasing function $\Theta\colon [0, +\infty) \to (0, +\infty)$, called \emph{growth function}.
Note that~\eqref{eq:Loeper-growth} corresponds to $\Theta$ constant, \eqref{eq:Miot-growth} to $\Theta(p)=p$ and~\eqref{eq:Holding-Miot-growth} to $\Theta(p)=p^{\frac1\alpha}$.
Also notice that the behavior of $\Theta(p)$ as $p\to+\infty$ only matters. 
We call such densities \emph{admissible} for the system~\eqref{eq:vp}, and we let
\begin{equation}
\label{eq:def_class}
\class^\Theta([0,T])
=
\set*{
f\in L^\infty([0,T];L^1(\R^{2d})) :
\rho_f\in L^\infty([0,T];Y^\Theta_\ul(\R^d))
}.
\end{equation}
Here and in the following, we let 
\begin{equation}
Y^\Theta_{\ul}(\R^d) 
= \set*{ f \in \bigcap_{p \in [1, +\infty)} L^p_{\ul}(\R^d) \colon \norm{f}_{Y^\Theta_{\ul}} = \sup_{p \in [1, +\infty)} \frac{\norm{f}_{L^p_{\ul}}}{\Theta(p)} < +\infty
}
\label{eq:def_yudo_ul} 
\end{equation}
be the \emph{uniformly-localized Yudovich space}, where, for $p\in[1,+\infty)$,
\begin{equation*} 
L^p_{\ul}(\R^d) 
= 
\set*{ 
f \in L^p_{\loc}(\R^d) \colon \norm{f}_{L^p_{\ul}} = \sup_{x \in \R^d} \norm{f}_{L^p(B_1(x))} < +\infty 
},
\end{equation*}
is the \emph{uniformly-localized $L^p$ space} on $\R^d$.
We also define the \emph{Yudovich space} $Y^\Theta(\R^d)$ as in~\eqref{eq:def_yudo_ul} by dropping the subscript `$\ul$' everywhere. 
These spaces were first introduced by Yudovich~\cite{Y95} to provide uniqueness of unbounded weak solutions of in\-com\-pres\-si\-ble inviscid $2$-dimensional Euler's equations.
We also refer to the recent works~\cites{CS21,CMZ19,T04,TTY10}. 

Following~\cites{L06,M16,HM18}, our starting point is the relation between the $L^p$ growth condition~\eqref{eq:Yudo-growth} and the continuity of the vector field~$E_f$, see \cref{res:relation_class_modulus} below.
Our result encodes the $\log$-Lipschitz regularity obtained in~\cite{L06}*{Lem.~3.1} following from~\eqref{eq:Loeper-growth}, as well as its more general version proved in~\cite{HM18}*{Lem.~2.1} concerning~\eqref{eq:Miot-growth} and~\eqref{eq:Holding-Miot-growth}. 
As for Euler's equations~\cite{CS21}, the main novelty here is that, once the spatial density~$\rho_f$ satisfies~\eqref{eq:Yudo-growth}, then we can explicitly express the (\emph{generalized}) \emph{modulus of continuity} of~$E_f$ depending on the chosen growth function~$\Theta$, namely, $\phi_\Theta\colon[0,+\infty)\to[0,+\infty)$ defined as
\begin{equation}
\phi_{\Theta} (r) =
\left\{\!\!
\begin{array}{ll}
0
&
\text{for}\ r=0,
\\[2ex]
r\, |\log r|\, \Theta(\abs*{\log r}) 
& 
\text{for}\ 
r \in (0, e^{-d-1}), 
\\[2ex] 
e^{-d-1} \,(d+1) \,\Theta(d+1) 
& 
\text{for}\ 
r \in [e^{-d-1}, +\infty)
\end{array}
\right.
\label{eq:phi_Theta}
\end{equation} 
(the choice of the constant $e^{-d-1}$ is irrelevant and is made for convenience only, see below).
With a slight abuse of notation, we set
\begin{equation*}
C^{0,\phi_\Theta}_b(\R^d;\R^d) = \set*{ E\in L^\infty(\R^d;\R^d)
:\ \sup_{x\neq y} \frac{\abs{E(x)- E(y)}}{\phi_\Theta (\abs{x-y})} < +\infty }. 
\end{equation*}

\begin{lemma}[Modulus of continuity]
\label{res:relation_class_modulus}
If $f\in\class^\Theta([0,T])$, then
\begin{equation*}
E_f\in L^\infty([0,T];C^{0,\phi_\Theta}_b(\R^d;\R^d)).
\end{equation*}   
\end{lemma}

The proof of  \cref{res:relation_class_modulus} revisits a classical strategy for proving Morrey's estimates for Riesz-type potential operators, see~\cite{MB02}*{Chap.~8} and~\cite{M16}*{Lem.~2.2} (for strictly related results see~\cite{GS15}*{Ths.~A and~B}).
Here we adopt the  elementary approach proposed in~\cite{CS21}*{Sec.~2}, generalizing the computations done in the $2$-dimensional case to any dimension.

\subsection{Weak solutions and transport equation}

A simple but quite crucial byproduct of \cref{res:relation_class_modulus} is that $fE_f\in L^\infty([0,T];L^1(\R^{2d};\R^d))$ whenever $f\in\class^\Theta([0,T])$.
This allows us to define weak solution of the system~\eqref{eq:vp} among admissible densities, as follows.  

\begin{definition}[Admissible weak solution]
\label{def:weak_sol_vp}
We say that $f\in \class^\Theta([0,T])$ is an \emph{admissible weak solution} of the system~\eqref{eq:vp} starting from the initial datum $f_0\in L^1(\R^{2d})$ if 
\begin{equation*}
\int_0^T\int_{\R^{2d}}
\big(\partial_t\psi+v\cdot\nabla_x\psi+E_f\cdot\nabla_v\psi\big)\,f\di x\di v\di t
=
-\int_{\R^{2d}}
\psi(0,\cdot)\,f_0\di x\di v
\end{equation*}
for any $\psi\in C^\infty_c([0,T)\times\R^{2d})$. 
\end{definition}
 
Due to the structure of the system~\eqref{eq:vp}, one is tempted to look for weak solutions $f\in\class^\Theta([0,T])$  transported along the flow of the vector field $b_f\colon[0,T]\times\R^{2d}\to\R^{2d}$,
\begin{equation}
\label{eq:trans_vec}
b_f(t,x,v)=(v,E_f(t,x))
\quad
\text{for}\
t\in[0,T],\ x,v\in\R^d.
\end{equation} 
The Cauchy problem corresponding to the vector field $b_f$ in~\eqref{eq:trans_vec} is in fact a second-order ODE that can be rewritten in  the form 
\begin{equation}
\label{eq:vp_ODE}
\left\{\!\!
\begin{array}{ll}
\dot X=V,
&
\text{for}\ t\in(0,T),
\\[1ex]
\dot V = E_f(t,X),
&
\text{for}\ t\in(0,T),
\\[1ex]
X(0)=x,\ V(0)=v,
\end{array}
\right.
\end{equation}
where $t\mapsto(X(t),V(t))$ is any flow line starting from the initial datum $(x,v)\in\R^{2d}$.
Since the modulus of continuity of~$b_f$ in~\eqref{eq:trans_vec} uniquely depends on~$\phi_\Theta$ in~\eqref{eq:phi_Theta}, which, in turn, only depends on the choice of~$\Theta$, here and in the rest of the paper we make the following 

\begin{assumption} \label{assumption}
The growth function $\Theta$ is such that $\phi_\Theta$ is continuous on $[0,+\infty)$.
\end{assumption}

Consequently, given a weak solution $f \in \class^\Theta([0,T])$, in virtue of \cref{res:relation_class_modulus} and Peano's Theorem, the Cauchy problem~\eqref{eq:vp_ODE} is well posed and admits a (classical) globally-defined, possibly non-unique, flow $\Gamma_f\colon[0,T]\times\R^{2d}\to\R^{2d}$.

\begin{definition}[Admissible Lagrangian weak solution]
\label{def:admissible_lagrangian}
We say that $f\in\class^\Theta([0,T])$ is an admissible \emph{Lagrangian} weak solution of the system~\eqref{eq:vp} starting from the initial datum $f_0\in L^1(\R^{2d})$ if $f$ is as in \cref{def:weak_sol_vp} and, moreover,
\begin{equation}
\label{eq:def_Lag_push}
f(t,\cdot)=(\Gamma_f(t,\cdot))_\#f_0
\quad
\text{for all}\ t\ge0,
\end{equation}
where $\Gamma_f$ is any flow solving the Cauchy problem~\eqref{eq:vp_ODE}.
\end{definition}
 
A natural way to ensure the well-posedness of the ODE in~\eqref{eq:vp_ODE} is to impose the \emph{Osgood condition} on the modulus of (spatial) continuity of~$b_f$ in~\eqref{eq:trans_vec}.
However, due to the special second-order structure of~\eqref{eq:vp_ODE}, such condition can be considerably relaxed.

\begin{theorem}[ODE well-posedness]
\label{res:vp_ODE}
Under \cref{assumption}, problem~\eqref{eq:vp_ODE} admits a globally-defined classical solution.
Moreover, if $\Phi_\Theta\colon[0,+\infty)\to[0,+\infty)$, given by
\begin{equation}
\label{eq:Phi_Theta}
\Phi_\Theta (r) 
= 
\int_0^r \phi_\Theta (s)  \di s
\quad
\text{for all}\ r\ge0,
\end{equation}  
satisfies
\begin{equation}
\label{eq:Phi_Theta_osgood}
\int_{0^+}\frac{\di r}{\sqrt{\Phi_\Theta(r)}}=+\infty,
\end{equation}
then the solution of problem~\eqref{eq:cauchy_problem} is unique and the induced flow is a measure-preserving homeomorphism on~$\R^{2d}$ at each time. 
\end{theorem}

Assumption~\eqref{eq:Phi_Theta_osgood} imposes the Osgood condition on $\sqrt{\Phi_\Theta}$ and can be seen as a second-order-type Osgood condition on~$\phi_\Theta$. 
Indeed, taking $d=1$, $X(0)=V(0)=0$ and $E_f(t,x)=\phi_\Theta(x)$ in~\eqref{eq:vp_ODE} for simplicity, we observe that 
\begin{equation*}
\frac{\di}{\di t} \frac{\dot X^2}2=\phi_\Theta(X)\,\dot X
\quad
\text{for}\ t\in(0,T),
\end{equation*}
so that, by integrating and changing variables, we get  
\begin{equation}
\label{eq:X2_unique}
\dot X^2(t)
=
2\int_0^t\phi_\Theta(X(s))\,\dot X(s)\di s
=
2\Phi_\Theta(X(t))
\quad
\text{for all}\ t\in(0,T).
\end{equation} 
Hence uniqueness of solutions of the ODE~\eqref{eq:vp_ODE} should follow as soon as
\begin{equation*}
\int_{0+}\frac{\dot X(t)\di t}{\sqrt{\Phi_\Theta(X(t))}}
=
\int_{0+}\frac{\di r}{\sqrt{\Phi_\Theta(r)}}
=+\infty,
\end{equation*}
leading to~\eqref{eq:Phi_Theta_osgood}. 
Note that~\eqref{eq:X2_unique} involves the (square of the) velocity $V=\dot X$ of the trajectory, besides its position~$X$, since in fact~$X$ solves a second-order ODE, namely, $\ddot X=E_f(t,X)$.
This explains why~\eqref{eq:Phi_Theta_osgood} should be seen as a second-order Osgood condition on the modulus of continuity of the vector field~$E_f$.

\subsection{Lagrangian stability}

Our first main result exploits the ODE well-posedness in \cref{res:vp_ODE} to provide stability of admissible Lagrangian weak solutions of the Vlasov--Poisson system~\eqref{eq:vp}, see \cref{res:main_lagrangian} below, generalizing~\cite{M16}*{Th.~1.1} and~\cite{HM18}*{Th.~1.9}.

Due to the physical meaning of the problem~\eqref{eq:vp} when $d=3$, we restrict our attention to non-negative densities $f\ge0$ and, up to (non-linearly) rescaling all estimates, we shall work with probability densities.
More precisely, we operate within the space of \emph{probability measures with finite $1$-moment} on $\R^{2d}$,  
\begin{equation*}
\prob_1(\R^{2d})
=
\set*{\mu\in\prob(\R^{2d}) : \int_{\R^{2d}}|p|\di\mu(p)<+\infty}.
\end{equation*}
Such space can be naturally endowed with the \emph{$1$-Wasserstein distance}, given by 
\begin{equation}
\label{eq:def_wass}
\wass_1(\mu_1,\mu_2)
=
\inf\set*{\int_{\R^{2d}\times\R^{2d}}|p-q|\di\pi(p,q) : \pi\in\plan(\mu_1,\mu_2)}
\end{equation}
for $\mu_1,\mu_2\in\prob_1(\R^{2d})$.
Here  
\begin{equation*}
\plan(\mu_1,\mu_2)
=
\set*{\pi\in\prob\big(\R^{2d}\times\R^{2d}\big) : (\proj_i)_\#\pi=\mu_i,\ i=1,2}
\end{equation*}
denotes the set of \emph{plans} (or \emph{couplings}) between~$\mu_1$ and~$\mu_2$, where $\proj_i\colon\R^{2d}\times\R^{2d}\to\R^{2d}$ is the projection on the $i$-th component.
As well-known~\cite{AGS08}, there exist \emph{optimal} plans $\pi\in\plan(\mu_1,\mu_2)$, i.e., plans attaining the infimum in~\eqref{eq:def_wass}, and the resulting \emph{$1$-Wasserstein space} $(\prob_1(\R^{2d}),\wass_1)$ is a complete and separable metric space.

\begin{theorem}[Lagrangian stability]
\label{res:main_lagrangian}
Assume that $\phi_\Theta$ is concave on $[0,+\infty)$ and $\Phi_\Theta$ satisfies~\eqref{eq:Phi_Theta_osgood}.
There is $\Omega_{\Theta,T}\colon[0,+\infty)\to[0,+\infty)$ continuous, with $\Omega_{\Theta,T}(0)=0$, satisfying the following property.
Let $i=1,2$ and let $f_i\in\class^\Theta([0,T])$ be a Lagrangian weak solution of the Vlasov--Poisson system~\eqref{eq:vp} starting from the initial datum $f_0^i\in L^1(\R^{2d})$.
If $\mu_0^i=f_0^i\,\mathscr L^{2d}\in\prob_1(\R^{2d})$, then also $\mu_i(t,\cdot)=f_i(t,\cdot)\,\mathscr L^{2d}\in\prob_1(\R^{2d})$ for all $t\in[0,T]$ and
\begin{equation*}
\sup_{t\in[0,T]}\wass_1(\mu_1(t,\cdot),\mu_2(t,\cdot))
\le 
\Omega_{\Theta,T}\big(
\wass_1(\mu_0^1,\mu_0^2\bigskip)
\big).
\end{equation*} 
In particular, if $f_0^1=f_0^2$, then also $f_1(t,\cdot)=f_2(t,\cdot)$ for all $t\in[0,T]$.
\end{theorem}

The function $\Omega_{\Theta,T}$ appearing in \cref{res:main_lagrangian} can be actually made more explicit and, basically, it depends on the inverse of the function $\Psi_{\Theta,\delta,c}\colon[0,+\infty)\to[0,+\infty)$,
\begin{equation*}
\Psi_{\Theta,\delta,c}(t)
=
\int_0^t\frac{\di s}{\delta+\sqrt{2c\,\Phi_\Theta(s)}}
\quad
\text{for all}\ t\ge0,
\end{equation*}
for suitably chosen parameters $\delta,c>0$.

The proof of \cref{res:main_lagrangian} follows the elementary strategy introduced in~\cite{CS21} for the well-posedness of $2$-dimensional Euler's equations (we also refer to recent applications of this method to \emph{transport--Stokes equations}~\cite{I23} and to systems of \emph{non-local} continuity equations~\cite{IS23}).
Basically, to control the distance between two Lagrangian weak solutions of the system~\eqref{eq:vp} in $\class^\Theta([0,T])$, in view of~\eqref{eq:def_Lag_push}, we just need to control the time evolution of the distance between the initial data along the flows of the corresponding Cauchy problem~\eqref{eq:vp_ODE} via a Gr\"onwall-type argument, exploiting both the stability of trajectories solving the associated ODE~\eqref{eq:vp_ODE} given by \cref{res:vp_ODE} and the modulus of continuity of the vector field provided by \cref{res:relation_class_modulus}.

Actually, our approach is more general and in fact provides stability of admissible Lagrangian weak solutions for a large family of system like~\eqref{eq:vp}.
More precisely, we can deal with \emph{generalized Vlasov--Poisson equations} of the form
\begin{equation} 
\label{eq:vp_gen_intro}
\left\{\!\!
\begin{array}{ll}
\partial_t f + F\cdot \nabla_x f + E_f\cdot\nabla_v f = 0 
& \text{in}\ (0,T) \times \R^{2d},
\\[2ex]
E_f(t,x)=\displaystyle\int_{\R^d}K(x,y)\,\rho_f(t,y)\di y
&
\text{for}\ t\in[0,T],\ x\in\R^d,
\\[2ex]
\rho_f(t,x)=\displaystyle\int_{\R^d}f(t,x,v)\di v 
&
\text{for}\ t\in[0,T],\ x\in\R^d,
\\[2ex]
f(0,\cdot) = f_0
& \text{on}\ \R^{2d},
\end{array}
\right.
\end{equation}
where $F\in L^\infty([0,T];C(\R^{2d};\R^d))$ satisfies
\begin{equation*}
\underset{{t\in[0,T]}}\esssup
\,
\abs*{F(t,x,v)- F(t,y,w)} \leq 
L\left[\abs{x-y}+\abs{v-w}\right] 
\quad
\text{for all}\
x,y,v,w\in\R^d 
\end{equation*}
for some $L\ge0$, and $K\colon\R^{2d}\to\R^d$ is any sufficiently well-behaved antisymmetric kernel.

The choice $F(t,x,v)=\frac{v}{\sqrt{1+|v|^2}}$ for $t\in[0,T]$ and $x,v\in\R^d$ in~\eqref{eq:vp_gen_intro} corresponds to the \emph{relativistic} Vlasov--Poisson equations.
The well-posedness theory in the relativistic framework is less understood.
For $d=3$ and only in the attractive case, global existence of solutions has been established in~\cites{GS85,GS01,HR07,KT08,W20} for radially symmetric initial data.
For both the attractive and the repulsive case, well-posedness---global for $d=2$ and only local for $d=3$---and propagation of regularity for general initial data have been recently obtained in~\cite{LS22} via propagation of velocity moments.

\subsection{Existence of Lagrangian solutions}

Our second main result provides existence of admissible Lagrangian weak solutions of the Vlasov--Poisson system~\eqref{eq:vp}, generalizing the constructions in~\cite{M16}*{Ths.~1.2 and~1.3} and~\cite{HM18}*{Prop.~1.14}.

\begin{theorem}[Existence]
\label{res:main_existence}
Let $d=2,3$.
Let $\theta\in Y^\Theta(\R^d)$ be such that 
\begin{equation}
\label{eq:theta_existence}
\theta\not\equiv0,
\quad
\theta\ge0
\quad
\text{and}
\quad
\int_{\R^d}(1\vee|x|)\,\theta(x)\di x<+\infty,
\end{equation} 
There exists a Lagrangian weak solution 
$f\in\class^\Theta([0,T])$ of the Vlasov--Poisson system~\eqref{eq:vp}, starting from the initial datum 
$$ f_0(x,v)=\frac{\mathbf 1_{(-\infty,0]}\left(|v|^2-\theta(x)^{\frac2d}\right)}{|B_1|\,\|\theta\|_{L^1}},
\quad
\text{for}\
x,v\in\R^d, $$
such that $f(t,\cdot)\,\mathscr L^{2d}\in\prob_1(\R^{2d})$ for all $t\in[0,T]$ and
\begin{equation*}
C\,\|\theta\|_{L^p}
\le
\|\rho_f\|_{L^\infty([0,T];L^p)}
\le 
C_T\|\theta\|_{L^p}
\quad
\text{for all}\ p\in\left[1,+\infty\right),
\end{equation*}
for some constants $C,C_T>0$, where $C_T$ depends on $T$.
\end{theorem}

The construction behind \cref{res:main_existence} builds upon the proofs of~\cite{M16}*{Ths.~1.2 and~1.3} and essentially applies the existence result proved in~\cite{LP91}*{Th.~1} to a suitable initial datum depending on the chosen function $\theta\in Y^\Theta(\R^d)$.

Note that any (non-zero) non-negative bounded and compactly supported function satisfies~\eqref{eq:theta_existence}.
Hence \cref{res:main_existence} becomes truly interesting if $\theta$ also satisfies
\begin{equation}
\label{eq:saturation}
\inf_{p\ge1}\frac{\|\theta\|_{L^p}}{\Theta(p)}>0,
\end{equation}
that is, the $L^p$ norm of $\theta$ grows as fast as $\Theta$.
In view of \cref{res:main_lagrangian}, we may restrict our attention only to growth functions $\Theta$ for which $\phi_\Theta$ is concave and condition~\eqref{eq:Phi_Theta_osgood} is met.
This is in fact the case for a countable family of growth functions of iterated-logarithmic type defined as follows.
For each $m\in\mathbb N$, we let $\Theta_m\colon[0,+\infty)\to[0,+\infty)$ be given by
\begin{equation*} 
\Theta_m(p)
=
\left\{\!\!
\begin{array}{ll}
p\,\abs{\log_1(p)}^2\, \abs{\log_2(p)}^2 \cdots \abs{\log_m(p)}^2 
&
\text{for}\ 
p\ge\exp_m(1),
\\[2ex]
\Theta_m(\exp_m(1))
&
\text{for}\ 
p\in[0,\exp_m(1)],
\end{array}
\right.
\end{equation*}
where $\exp_0(1)=1$ and $\exp_{m+1}(1)=e^{\exp_{m}(1)}$ recursively, and
\begin{equation}
\label{eq:logol}
\log_m=
\left\{\!\!
\begin{array}{ll}
\mathrm{id}
&
\text{for}\ m=0
\\[1ex]
\underbrace{\log\log\cdots\log}_{(m-1)\ \text{times}}|\log|
&
\text{for}\ 
m\ge1.
\end{array}
\right.
\end{equation}

\begin{proposition}[Saturation of $\Theta_m$]
\label{res:saturation}
For each $m\in\mathbb N_0$, $\phi_{\Theta_m}$ is concave, $\Phi_{\Theta_m}$ satisfies~\eqref{eq:Phi_Theta_osgood} and there is $\theta_m\in Y^{\Theta_m}(\R^d)$ with compact support satisfying~\eqref{eq:theta_existence} and~\eqref{eq:saturation}. 
\end{proposition}

\cref{res:main_existence}, together with \cref{res:saturation}, yield that the class of admissible Lagrangian weak solutions considered in \cref{res:main_lagrangian} is non-empty for $d\in\set*{2,3}$ and $\Theta=\Theta_m$ for some $m\in\mathbb N_0$.
When $m=0$, our results embed the example given in the proof of~\cite{M16}*{Th.~1.3}.
Actually, the functions $\theta_m$ in \cref{res:saturation} are modeled on a well-known example due to Yudovich (see~\cite{Y95}*{Eq.~(3.7)}, \cite{T04}*{Rem.~1(i)} and the discussion around~\cite{CS21}*{Eq.~(1.12)}) concerning $2$-dimensional Euler equations in vorticity form.

\subsection{Organization of the paper}

In \cref{sec:abstract_lagrangian_stab} we provide an abstract approach to achieve the well-posedness of the Cauchy problem~\eqref{eq:vp_ODE} and the stability of admissible Lagrangian weak solutions of the system~\eqref{eq:vp}, considering the generalized Vlasov--Poisson equations~\eqref{eq:vp_gen_intro}.
We refer the reader to \cref{res:cauchy_problem} and \cref{res:lagrangian_stab_vp_gen}, respectively.
In \cref{sec:proofs}, we detail the proofs of the results presented above.

\section{Lagrangian stability for a generalized Vlasov--Poisson system}
\label{sec:abstract_lagrangian_stab}

In this section, we provide an abstract approach to obtain stability properties for Lagrangian solutions of (a generalized version of) the Vlasov--Poisson system~\eqref{eq:vp}.
Our stability result is stated in \cref{res:lagrangian_stab_vp_gen} and exploits the well-posedness of the corresponding second-order Cauchy problem provided by \cref{res:cauchy_problem}. 

\subsection{Notation}

Throughout this section, we consider
\begin{equation}
\label{eq:gronwall_phi}
\phi\in C([0,+\infty);[0,+\infty)),
\quad
\text{with}\ 
\phi(t)>0\ \text{for}\ t>0.
\end{equation}
We also let $\Phi\colon[0,+\infty)\to[0,+\infty)$ be given by
\begin{equation}
\label{eq:gronwall_Phi}
\Phi(t)=\int_0^t\phi(s)\di s
\quad
\text{for all}\ t\ge0.
\end{equation} 
Note that $\Phi$ is a non-negative and non-decreasing  $C^1$ function. 
For certain results we will also assume that $\Phi$ satisfies the additional condition
\begin{equation}
\label{eq:osgood-2}
\int_{0^+}\frac{\di t}{\sqrt{\Phi(t)}}=+\infty,
\end{equation}
i.e., the function $\sqrt{\Phi}$ satisfies the Osgood condition.
Clearly, condition~\eqref{eq:osgood-2} implies that $\phi(0)=0$. 
Given $\delta,c>0$, we also define the function $\Psi_{\delta,c}\colon[0,+\infty)\to[0,+\infty)$ by 
\begin{equation*}
\Psi_{\delta,c}(t)
=
\int_0^t\frac{\di s}{\delta+\sqrt{2c\,\Phi(s)}}
\quad
\text{for all}\ t\ge0.
\end{equation*}
To keep the notation short, we set $\Psi_{\delta}=\Psi_{\delta,1}$.
Note that $\Psi_{\delta,c}$ is a non-negative and strictly increasing $C^1$ function with bounded derivative.
In particular, $\Psi_{\delta,c}$ is invertible, with continuous and strictly-increasing inverse.
Note that, if~\eqref{eq:osgood-2} is assumed, then
\begin{equation*}
\lim_{\delta\to0^+}\Psi_{\delta,c}(t)=+\infty
\quad
\text{and}
\quad
\lim_{\delta\to0^+}\Psi_{\delta,c}^{-1}(t)=0
\quad
\text{for all}\ t,c>0.
\end{equation*}

\subsection{Second-order Gr\texorpdfstring{\"o}{o}nwall's inequality}
We begin with the following result, which may be considered as a Gr\"onwall-type lemma for a second-order differential inequality.

\begin{lemma}
[Gr\"onwall]
\label{res:gronwall-2}
Let $u\in W^{2,\infty}([0,T])$ be such that $u,u'\ge0$.
If
\begin{equation}
\label{eq:gronwall-2}
u''\le cu'+\phi(u)
\quad
\text{a.e.\ in}\ [0,T] 
\end{equation}
for some $c>0$ and $u'(0)\le\delta$ for some $\delta>0$, then
\begin{equation*}
u'(t)
\le 
e^{ct}\big(
\delta+\sqrt{2\Phi(u(t))}
\big)
\quad\text{and}\quad
u(t)
\le
\Psi_\delta^{-1}
\left(
\Psi_\delta(u(0))+e^{ct}-1
\right)
\end{equation*}
for all $t\in[0,T]$.
\end{lemma}

\begin{proof}
Multiplying~\eqref{eq:gronwall-2} by $u'\ge0$, we get 
\begin{equation*}
\frac{\di}{\di t}\left[(u')^2\right]
\le 
2c(u')^2+2\phi(u)u'
\quad
\text{a.e.\ in}\ [0,T].
\end{equation*}
Integrating and changing variables, we can estimate
\begin{align*}
(u'(t))^2
&\le
(u'(0))^2
+
2\Phi(u(t))-2\Phi(u(0))
+
2c\int_0^t(u'(s))^2\di s
\\
&\le
\delta^2
+
2\Phi(u(t))
+
2c\int_0^t(u'(s))^2\di s
\end{align*} 
for all $t\in[0,T]$. 
Since $t\mapsto\Phi(u(t))$ is non-decreasing, by Gr\"onwall's inequality we get
\begin{equation*}
(u'(t))^2
\le 
e^{2ct}\left(\delta^2
+
2\Phi(u(t))\right)
\quad
\text{for all}\ t\in[0,T],
\end{equation*}  
so that 
\begin{equation*}
\frac{u'(t)}{\delta
+
\sqrt{2\Phi(u(t))}}
\le 
e^{ct}
\quad
\text{for all}\ t\in[0,T].
\end{equation*}
Integrating the above inequality, we conclude that 
\begin{equation*}
\Psi_\delta(u(t))-\Psi_\delta(u(0))
\le
e^{ct}-1 
\quad
\text{for all}\ t\in[0,T],
\end{equation*}
from which the conclusion follows immediately.
\end{proof}

\subsection{Second-order Cauchy problem}
\label{subsec:cauchy}

We let $b\colon[0,T]\times\R^{2d}\to\R^{2d}$ be given by
\begin{equation}
\label{eq:b_vector_field}
b(t,x,v)=\big(F(t,x,v),E(t,x)\big)
\quad
\text{for}\
t\in[0,T],\ x,v\in\R^d,
\end{equation}
where $E\in L^\infty([0,T];C_b(\R^{d};\R^d))$ satisfies
\begin{equation}
\underset{{t\in[0,T]}}\esssup
\,
\abs*{E(t,x) - E(t,y)} \leq \phi(\abs{x-y})
\quad 
\text{for all}\ x,y \in \R^d
\label{eq:E_cont}
\end{equation}
with $\phi$ as in~\eqref{eq:gronwall_phi} and $F\in L^\infty([0,T];C(\R^{2d};\R^d))$ satisfies
\begin{equation}
\underset{{t\in[0,T]}}\esssup
\,
\abs*{F(t,x,v)- F(t,y,w)} \leq 
L\left[\abs{x-y}+\abs{v-w}\right] 
\quad
\text{for all}\
x,y,v,w\in\R^d
\label{eq:F_lip}  
\end{equation}
for some fixed $L\in[0,+\infty)$.
For any given $x,v\in\R^d$, we consider the Cauchy problem
\begin{equation}
\label{eq:cauchy_problem}
\left\{\!\!
\begin{array}{ll}
\dot\gamma_{x,v}=b(t,\gamma_{x,v}),
&
\text{for}\ t\in(0,T),
\\[1ex]
\gamma(0)=(x,v).
\end{array}
\right.
\end{equation}
Note that~\eqref{eq:cauchy_problem} is in fact a second-order Cauchy problem and can be rewritten as
\begin{equation}
\label{eq:cauchy_problem_XV}
\left\{\!\!
\begin{array}{ll}
\dot X=F(t,X,V),
&
\text{for}\ t\in(0,T),
\\[1ex]
\dot V = E(t,X),
&
\text{for}\ t\in(0,T),
\\[1ex]
X(0)=x,\ V(0)=v,
\end{array}
\right.
\end{equation}
denoting $\gamma_{x,v}(t)=(X(t,x,v),V(t,x,v))$ for $t\in[0,T]$, $x,v\in\R^d$.

\begin{theorem}[ODE well-posedness]
\label{res:cauchy_problem}
Problem~\eqref{eq:cauchy_problem} admits a globally-defined classical solution $\gamma_{x,v}\in W^{1,\infty}([0,T];\R^{2d})$ for all $x,v\in\R^d$.
Moreover, if~$\Phi$ in~\eqref{eq:gronwall_Phi} satisfies condition~\eqref{eq:osgood-2}, then the solution of~\eqref{eq:cauchy_problem} is unique for all $x,v\in\R^d$. 
Finally, letting 
\begin{equation*}
\Gamma\colon[0,T]\times\R^{2d}\to\R^{2d}
\end{equation*}
with $\Gamma(t,x,v)=\gamma_{x,v}(t)$ for $t\in[0,T]$, $x,v\in\R^d$, be the associated flow map, if $\divergence_x F=0$, then $\Gamma(t,\cdot)$ is a measure-preserving homeomorphism on~$\R^{2d}$ for all $t\in[0,T]$. 
\end{theorem}

Since $b\in L^\infty([0,T];C(\R^{2d};\R^{2d}))$ has at most linear growth, the first part of \cref{res:cauchy_problem} concerning the global existence of at least one solution of~\eqref{eq:cauchy_problem} follows by standard ODE Theory (namely, by Peano's Theorem and Gr\"onwall's inequality). Instead, the validity of the second part of \cref{res:cauchy_problem} concerning the uniqueness of the solution of~\eqref{eq:cauchy_problem} and the measure-preserving property of the associated flow map follows from the following result. 

\begin{proposition}[ODE stability]
\label{res:stability_cauchy}
Let $i=1,2$, let $b_i=(F_i,E_i)$ be as in~\eqref{eq:b_vector_field}, with $E_i\in L^\infty([0,T];C_b(\R^{d};\R^d))$ satisfying~\eqref{eq:E_cont} and $F_i\in L^\infty([0,T];C(\R^{2d};\R^d))$ satisfying~\eqref{eq:F_lip}, and let $\gamma_i=(X_i,V_i)\in W^{1,\infty}([0,T];\R^{2d})$ be a solution of~\eqref{eq:cauchy_problem} with initial condition $(x_i, v_i)\in\R^{2d}$.
If  
\begin{equation*}
L\abs{x_1-x_2} 
+ 
L\abs{v_1- v_2} 
+ 
L\norm{E_1- E_2}_{L^\infty(C)}
+ 
\norm{F_1-F_2}_{L^\infty(C)}
\le
\delta
\end{equation*}
for some $\delta>0$, then 
\begin{equation*}
\begin{split}
\norm*{\gamma_1-\gamma_2}_{L^\infty}
&\le 
\abs{v_1-v_2} 
+ 
\norm{E_1-E_2}_{L^\infty}  
+
\Psi_{\delta,L}^{-1}
\left(
\Psi_{\delta,L}(|x_1-x_2|)
+
e^{LT}-1
\right)
\\
&\quad
+ 
T\phi\left(
\Psi_{\delta,L}^{-1}
\left(
\Psi_{\delta,L}(|x_1-x_2|)
+
e^{LT}-1
\right)
\right).
\end{split}
\end{equation*}
\end{proposition}

\begin{proof}
In the following, we drop the spatial variables to keep the notation short.
In virtue of~\eqref{eq:F_lip} and~\eqref{eq:cauchy_problem_XV}, we can estimate
\begin{equation}
\label{eq:stimo_X-X}
\begin{split}
|X_1(t) &- X_2(t)| 
\leq 
\abs{x_1-x_2} + \int_0^t \abs{F_1(s, X_1(s), V_1(s)) -F_2(s, X_2(s), V_2(s)) }  \di s 
\\ 
& \leq 
\abs{x_1-x_2} + \int_0^t \abs{F_1(s, X_1(s), V_1(s)) -F_1(s, X_2(s), V_2(s)) }  \di s 
\\ 
&\quad+  
\int_0^t \abs{F_1(s, X_2(s), V_2(s)) -F_2(s, X_2(s), V_2(s)) }  \di s 
\\ 
& \leq 
\abs{x_1-x_2} + L\int_0^t \abs{X_1(s) - X_2(s)} \di s
+
L\int_0^t \abs{V_1(s)- V_2(s)}  \di s 
+ 
t \norm{F_1-F_2}_{L^\infty} 
\end{split}
\end{equation}
for all $t\in[0,T]$.
Because of~\eqref{eq:E_cont} and again of~\eqref{eq:cauchy_problem_XV}, we can also estimate
\begin{equation}
\label{eq:stimo_V-V}
\begin{split}
|V_1(s) &- V_2(s)|  
\leq 
\abs{v_1 -v_2}
+ 
\int_0^s \abs{E_1(r, X_1(r)) - E_2(r, X_2(r))}  \di r 
\\ 
& \leq 
\abs{v_1-v_2}  + \int_0^s \abs{E_1(r, X_1(r)) - E_1(r, X_2(r))}  \di r 
\\
&\quad +  
\int_0^s \abs{E_1(r, X_2(r)) - E_2(r, X_2(r))}  \di r 
\\ 
& \leq 
\abs{v_1-v_2} 
+ 
\norm{E_1-E_2}_{L^\infty}  
+ 
\int_0^s \phi (\abs{X_1(r) - X_2(r) })  \di r 
\end{split}
\end{equation}
for all $s\in[0,T]$.
Therefore, we obtain that 
\begin{equation}
\label{eq:definisco_u}
\begin{split}
\abs{X_1(t) - X_2(t)} 
& \leq 
\abs{x_1-x_2} + t\left[ 
L\abs{v_1 -v_2} 
+  
L\norm{E_1-E_2}_{L^\infty}
+
\norm{F_1-F_2}_{L^\infty} 
\right]  
\\ 
& \quad + 
L\int_0^t \abs{X_1(s) - X_2(s)} \di s + L\int_0^t \int_0^s \phi(\abs{X_1(r)- X_2(r)})  \di r  \di s 
\end{split} 
\end{equation}
for all $t\in[0,T]$. 
Letting $u\in W^{2,\infty}([0,T])$ be the function in the right-hand side of~\eqref{eq:definisco_u}, we observe that $u\ge0$, $u(0)=|x_1-x_2|$,
\begin{equation}
\label{eq:calcolo_u'}
\begin{split}
u'(t) & =
L 
\abs{v_1-v_2} 
+ L\norm{E_1-E_2}_{L^\infty} 
+ \norm{F_1- F_2}_{L^\infty} 
+ L\abs{X_1(t)- X_2(t)}
\\ 
& \quad  
+ L\int_0^t \phi(\abs{X_1(s)- X_2(s)})  \di s,
\end{split}
\end{equation}
for all $t\in[0,T]$ and so, in particular, 
\begin{equation*}
u'(0)=
L\abs{x_1-x_2} 
+ L\abs{v_1- v_2} 
+ L\norm{E_1- E_2}_{L^\infty}
+ \norm{F_1-F_2}_{L^\infty}\le\delta.
\end{equation*}
We also observe that 
\begin{equation}
\label{eq:calcolo_u''}
u''(t)
\le 
L\abs{\dot{X}_1(t) - \dot{X}_2(t)} + L\phi(\abs{X_1(t) - X_2(t)})
\quad
\text{for a.e.}\ t\in[0,T].
\end{equation}
We now estimate the right-hand side of~\eqref{eq:calcolo_u''} in terms of~$u$. 
Exploiting~\eqref{eq:F_lip}, \eqref{eq:cauchy_problem_XV} and the estimate in~\eqref{eq:stimo_V-V}, we have
\begin{align*}
|\dot{X}_1(t)- \dot{X}_2(t)|  
&= \abs{F_1(t, X_1(t), V_1(t)) - F_2(t, X_2(t), V_2(t)) }
\\ 
& \leq 
\norm{F_1(t) -F_2(t)}_{C} + L\abs{X_1(t) - X_2(t)} + L\abs{V_1(t) - V_2(t)}
\\ 
& \leq 
\norm{F_1 -F_2}_{L^\infty} + L\abs{X_1(t) - X_2(t)}
+ L\abs{v_1-v_2} 
\\ 
&\quad + 
L\norm{E_1 -E_2}_{L^\infty} 
+ 
L\int_0^t \phi(\abs{ X_1(s) -  X_2(s)})  \di s   
\\ 
& = u'(t) 
\end{align*}
for all $t\in[0,T]$ in virtue of~\eqref{eq:calcolo_u'}.
We thus get that $u$ satisfies
\begin{equation*}
u''\le Lu'+L\phi(u)
\quad
\text{a.e.\ in}\ [0,T],
\end{equation*}
as in~\eqref{eq:gronwall-2} in \cref{res:gronwall-2}, from which we immediately get that
\begin{equation*}
|X_1(t)-X_2(t)|
\le 
\Psi_{\delta,L}^{-1}
\left(
\Psi_{\delta,L}(|x_1-x_2|)
+
e^{Lt}-1
\right)
\end{equation*} 
for all $t\in[0,T]$. 
Consequently, by~\eqref{eq:stimo_V-V}, we also find that
\begin{equation*}
\begin{split}
|V_1(t)-V_2(t)|
\le 
\abs{v_1-v_2} 
+ 
\norm{E_1-E_2}_{L^\infty}  
+ 
t\,\phi\left(
\Psi_{\delta,L}^{-1}
\left(
\Psi_{\delta,L}(|x_1-x_2|)
+
e^{LT}-1
\right)
\right)
\end{split}
\end{equation*}
for all $t\in[0,T]$, from which the conclusion immediately follows.
\end{proof}

From \cref{res:stability_cauchy}, we plainly deduce the following approximation result.

\begin{corollary}[ODE convergence] \label{res:approximation_cauchy}
Let $n\in\mathbb N$, let $b=(F,E),b_n=(F_n,E_n)$ be as in~\eqref{eq:b_vector_field}, with $E,E_n\in L^\infty([0,T];C_b(\R^{d};\R^d))$ satisfying~\eqref{eq:E_cont} and $F,F_n\in L^\infty([0,T];C(\R^{2d};\R^d))$ satisfying~\eqref{eq:F_lip}, and let $\gamma_n=(X_n,V_n)\in W^{1,\infty}([0,T];\R^{2d})$ be a solution of~\eqref{eq:cauchy_problem} with initial condition $(x,v)\in\R^{2d}$.
If $\Phi$ in~\eqref{eq:gronwall_Phi} satisfies~\eqref{eq:osgood-2} and
\begin{equation}
\label{eq:approximation_b}
\lim_{n\to+\infty}\norm*{b_n-b}_{L^\infty}=0,
\end{equation} 
then $(\gamma_n)_{n\in\mathbb N}$ is a Cauchy sequence in $C([0,T]\times \R^{2d})$, and each of its limit points $\gamma=(X,V)$ is a solution of~\eqref{eq:cauchy_problem} relative to $b=(F,E)$ with initial condition $(x,v)$. 
\end{corollary}

\begin{proof}
By \cref{res:stability_cauchy}, we immediately infer that 
\begin{equation*}
\begin{split}
\norm*{\gamma_m-\gamma_n}_{L^\infty}
&\le 
\delta_{m,n}  
+
\Psi_{\delta_{m,n},L}^{-1}(e^{LT}-1)
+ 
T\phi\left(
\Psi_{\delta_{m,n},L}^{-1}(e^{LT}-1)
\right).
\end{split}
\end{equation*}
for all $m,n\in\mathbb N$, where 
\begin{equation*}
\delta_{m,n}
=
\norm{E_m- E_n}_{L^\infty}
+ \norm{F_m-F_n}_{L^\infty}+\tfrac1m+\tfrac1n.
\end{equation*}
Since $\delta_{m,n}\to0^+$ as $m,n\to+\infty$, by~\eqref{eq:osgood-2} we infer that $\Psi^{-1}_{\delta_{m,n},L}(e^{LT}-1)\to0^+$ as $m,n\to+\infty$, easily yielding the conclusion.
\end{proof}

We are now ready to prove \cref{res:cauchy_problem}.

\begin{proof}[Proof of \cref{res:cauchy_problem}]
We just need to deal with the second part of the statement concerning the uniqueness of the solution of~\eqref{eq:cauchy_problem} and the measure-preserving property of the associated flow map.
The uniqueness part is an immediate consequence of \cref{res:stability_cauchy}.
Indeed, if $\gamma_1$ and $\gamma_2$ are two solutions of~\eqref{eq:cauchy_problem} relative to $b$ starting from the same initial datum $(x,v)$, with $x,v\in\R^n$, then \cref{res:stability_cauchy} implies that 
\begin{equation*}
\begin{split}
\norm*{\gamma_1-\gamma_2}_{L^\infty}
&\le 
\Psi_{\delta,L}^{-1}(e^{LT}-1)
+ 
T\phi\left(
\Psi_{\delta,L}^{-1}(e^{LT}-1)
\right)
\end{split}
\end{equation*}
for all $\delta>0$.
Since $\Psi_{\delta,L}^{-1}(e^{LT}-1)\to0^+$ as $\delta\to0^+$, we get $\gamma_1=\gamma_2$.
The measure-preserving property of the associated flow map, instead, follows from an approximation argument and \cref{res:approximation_cauchy}.
We leave the simple details to the reader.
\end{proof}

\subsection{Generalized Vlasov--Poisson system}

From now on, we fix a measurable function $K\colon\R^{2d}\to\R^d$, that we call \emph{kernel}, which is assumed to be antisymmetric, i.e., $K(x,y)=-K(x,y)$ for a.e.\ $x,y\in\R^d$.
We thus consider the associated Vlasov--Poisson-type system
\begin{equation} \label{eq:vp_gen}
\left\{\!\!
\begin{array}{ll}
\partial_t f + F\cdot \nabla_x f + E_f\cdot\nabla_v f = 0 
& \text{in}\ (0,T) \times \R^{2d},
\\[2ex]
E_f(t,x)=\displaystyle\int_{\R^d}K(x,y)\,\rho_f(t,y)\di y
&
\text{for}\ t\in[0,T],\ x\in\R^d,
\\[2ex]
\rho_f(t,x)=\displaystyle\int_{\R^d}f(t,x,v)\di v 
&
\text{for}\ t\in[0,T],\ x\in\R^d,
\\[2ex]
f(0,\cdot) = f_0
& \text{on}\ \R^{2d},
\end{array}
\right.
\end{equation}
where the unknown density is $f\in L^\infty([0,T];L^1(\R^{2d}))$ and the initial datum is $f_0\in L^1(\R^{2d})$.
The function $F\in L^\infty([0,T];C(\R^{2d};\R^d))$ in the first line of~\eqref{eq:vp_gen}  always satisfies~\eqref{eq:F_lip}, and may be additionally assumed to satisfy $\divergence_x F=0$. 
If $F(t,x,v)=v$, then~\eqref{eq:vp_gen} reduces to the classical Vlasov--Poisson system, while, if $F(t,x,v)=\frac{v}{\sqrt{1+|v|^2}}$, then~\eqref{eq:vp_gen} becomes the relativistic Vlasov--Poisson system.

\begin{definition}[Weak $\phi$-solution]
\label{def:weak_sol}
We say that $f \in L^\infty([0,T];L^1(\R^{2d}))$ is a \emph{weak $\phi$-solution} of~\eqref{eq:vp_gen} with initial datum $f_0 \in L^1(\R^{2d})$ if
\begin{equation}
\label{eq:K_rho_bounded}
(t,x)\mapsto\int_{\R^d}|K(x,z)|\,|\rho_f(t,z)|\di z
\in 
L^\infty([0,T]\times\R^d),
\end{equation}
\begin{equation}
\label{eq:K_rho_cont}
\underset{t\in[0,T]}\esssup
\int_{\R^d}|K(x,z)-K(y,z)|\,|\rho_f(t,z)|\di z
\le 
\phi(|x-y|)
\quad
\text{for all}\ x,y\in\R^d
\end{equation}
and 
\begin{equation}
\label{eq:def_phi-sol}
\int_0^T\int_{\R^{2d}}
\big(\partial_t\psi+F\cdot\nabla_x\psi+E_f\cdot\nabla_v\psi\big)\,f\di x\di v\di t
=
-\int_{\R^{2d}}
\psi(0,\cdot)\,f_0\di x\di v
\end{equation}
for all $\psi\in C^\infty_c([0,T)\times\R^{2d})$, where $E_f, \rho_f$ are as in~\eqref{eq:vp_gen}. 
\end{definition}

Note that, if $f$ is a weak $\phi$-solution of~\eqref{eq:vp_gen} as in \cref{def:weak_sol}, then~\eqref{eq:K_rho_bounded} and~\eqref{eq:K_rho_cont} lead to $E_f\in L^\infty([0,T];C_b(\R^{d};\R^d))$  satisfying~\eqref{eq:E_cont}.
In particular, the equation~\eqref{eq:def_phi-sol} is well defined, since $fE_f\in L^\infty([0,T];L^1(\R^{2d};\R^d))$ thanks to~\eqref{eq:K_rho_bounded}. 

\begin{definition}[Lagrangian weak $\phi$-solution]
\label{def:lagrangian}
We say that $f \in L^\infty([0,T];L^1(\R^{2d}))$ is a \emph{Lagrangian weak $\phi$-solution} of~\eqref{eq:vp_gen} with initial datum $f_0 \in L^1(\R^{2d})$ if $f$ is a weak $\phi$-solution of~\eqref{eq:vp_gen} as in \cref{def:weak_sol} and, moreover, 
\begin{equation}
\label{eq:def_lagrangian}
f(t,\cdot) = \Gamma(t,\cdot)_\#f_0
\quad
\text{for all}\ t\in[0,T],
\end{equation}
where $\Gamma$ is any flow map associated to the Cauchy problem~\eqref{eq:cauchy_problem} with $b=(F,E)$.
\end{definition}

The following result collects two basic features of Lagrangian weak $\phi$-solutions of~\eqref{eq:vp_gen} that will be useful in the sequel.

\begin{lemma}[Sign and moment preservation]
Assume $\divergence_x F=0$ and $\Phi$ in~\eqref{eq:gronwall_Phi} satisfies~\eqref{eq:osgood-2}.
Let $f \in L^\infty([0,T];L^1(\R^{2d}))$ be a Lagrangian weak $\phi$-solution of~\eqref{eq:vp_gen} with initial datum $f_0 \in L^1(\R^{2d})$.
If $f_0\ge0$, then also $f(t,\cdot)\ge0$ for all $t\in[0,T]$.
Moreover, if $\mu_0=f_0\,\mathscr L^{2d}\in\prob_1(\R^{2d})$, then also $\mu(t,\cdot)=f(t,\cdot)\,\mathscr L^{2d}\in\prob_1(\R^{2d})$ for all $t\in[0,T]$. 
\end{lemma}

\begin{proof}
Fix $t\in[0,T]$.
Since $\Gamma(t,\cdot)$ is a measure-preserving homeomorphism by \cref{res:stability_cauchy}, then from~\eqref{eq:def_lagrangian} we easily deduce that
\begin{equation*}
\begin{split}
\mathscr{L}^{2d} \left(\set*{ z\in\R^{2d} : f(t,z) <0 }\right) 
&=
\mathscr{L}^{2d} \left(\set*{ z\in\R^{2d} : f(t,\Gamma(t,z)) < 0 } \right)
\\
&= 
\mathscr{L}^{2d} \left( \set*{ z\in\R^{2d} : f_0(z) <0  } \right)= 0,  
\end{split}
\end{equation*}
so that $f(t,\cdot)\ge0$. 
In addition, if 
\begin{equation*}
\int_{\R^{2d}}|z|\di\mu_0(z)
=
\int_{\R^{2d}}|z|\,f_0(z)\di z
<+\infty,
\end{equation*}
then again by~\eqref{eq:def_lagrangian} we get 
\begin{equation*}
\int_{\R^{2d}}|z|\di\mu(t,z)
=
\int_{\R^{2d}}|z|\,f(t,z)\di z
=
\int_{\R^{2d}}|\Gamma(t,z)|\,f_0(z)\di z<+\infty,
\end{equation*}
since $|\Gamma(t,z)|\le C|z|e^{CT}$ for all $t\in[0,T]$ and $z\in\R^{2d}$, for some $C>0$ depending on $\|E_f\|_{L^\infty}$ and $\norm{F}_{L^\infty(\Lip)}$ only, by standard ODE Theory, in virtue of~\eqref{eq:F_lip} and~\eqref{eq:K_rho_bounded}.
\end{proof}

We can now state and prove the main result of this section, providing a stability property for Lagrangian weak $\phi$-solutions of the Vlasov--Poisson-type system~\eqref{eq:vp_gen}.
The proof of \cref{res:lagrangian_stab_vp_gen} adopts the elementary point of view of~\cite{CS21} and extends the approaches exploited in the proofs of~\cite{M16}*{Th.~1.1} and~\cite{HM18}*{Th.~1.9}. 

\begin{theorem}[Lagrangian stability]
\label{res:lagrangian_stab_vp_gen}
Let $i=1,2$, let $\mu_i\in L^\infty([0,T];\prob_1(\R^{2d}))$ be such that $\mu_i=f_i\,\mathscr L^{2d}$, where $f_i\in L^\infty([0,T];L^1(\R^{2d}))$ is a Lagrangian weak $\phi$-solution of~\eqref{eq:vp_gen},  relative to $(F_i,E_i)$, $E_i=E_{f_i}$, with $F_i\in L^\infty([0,T];C(\R^d;\R^d))$ satisfying~\eqref{eq:F_lip} for some $L\in[1,+\infty)$ and $\divergence_xF_i=0$, with initial datum $f_0^i\in L^1(\R^{2d})$. 
Assume that $\phi$ in~\eqref{eq:gronwall_phi} is concave and $\Phi$ in~\eqref{eq:gronwall_Phi} satisfies~\eqref{eq:osgood-2}.
If
\begin{equation*}
2L\,\wass_1(\mu_0^1,\mu_0^2)+\|F_1-F_2\|_{L^\infty}
<\delta
\end{equation*}
for some $\delta>0$, then
\begin{equation*}
\begin{split}
\wass_1(\mu_1(t,\cdot),\mu_2(t,\cdot))
&\le 
\Psi_{\delta,2L}^{-1}
\big(
\Psi_{\delta,2L}(\wass_1(\mu_0^1,\mu_0^2))+e^{Lt}-1
\big)
\\
&\quad+
e^{Lt}
\left(
\delta+\sqrt{4L\Phi\big(\Psi_{\delta,2L}^{-1}
\big(
\Psi_{\delta,2L}(\wass_1(\mu_0^1,\mu_0^2))+e^{Lt}-1
\big)
\big)}
\right)
\end{split}
\end{equation*}
for all $t\in[0,T]$.
In particular, if $f_0^1=f_0^2$ and $F_1=F_2$, then $f_1=f_2$.
\end{theorem}

\begin{proof}
Let $\pi_0\in\plan(\mu_0^1,\mu_0^2)$ be an optimal plan.
By \cref{def:lagrangian}, we can write $\mu_i(t,\cdot)=\Gamma_i(t,\cdot)_\#\mu_0^i$ for $t\in[0,T]$ and $i=1,2$, so that 
\begin{equation}
\label{eq:plan_t}
\pi(t,\cdot)=(\Gamma_1(t,\proj_1),\Gamma_2(t,\proj_2))_\#\pi_0
\in
\plan(\mu_1(t,\cdot),\mu_2(t,\cdot))
\end{equation}
for all $t\in[0,T]$.
Since $\Gamma_i=(X_i,V_i)$, $i=1,2$, we define
\begin{align}
\XX(t)
&=
\int_{\R^{2d}\times\R^{2d}}
|X_1(t,p)-X_2(t,q)|\di\pi_0(p,q)
\label{eq:def_XX}
\\
\VV(t)
&=
\int_{\R^{2d}\times\R^{2d}}
|V_1(t,p)-V_2(t,q)|\di\pi_0(p,q)
\nonumber
\end{align}
for all $t\in[0,T]$, where $p=(x,v)$ and $q=(y,w)$. 
Arguing as in~\eqref{eq:stimo_X-X}, we can estimate
\begin{equation*}
\begin{split}
|X_1(t,p)-X_2(t,q)|
&\le 
|x-y|
+
L\int_0^t|X_1(s,p)-X_2(s,q)|\di s
+
L\int_0^t|V_1(s,p)-V_2(s,q)|\di s
\\
&\quad +
t\|F_1-F_2\|_{L^\infty}
\end{split}
\end{equation*}
for all $t\in[0,T]$, so that 
\begin{equation*}
\begin{split}
\XX(t)
&\le 
\int_{\R^{2d}\times\R^{2d}}|x-y|\di\pi_0(p,q)
+
t\|F_1-F_2\|_{L^\infty}
+
L\int_0^t\XX(s)\di s
+
L\int_0^t\VV(s)\di s
\\
&\le 
\wass_1(\mu_0^1,\mu_0^2)
+
t\|F_1-F_2\|_{L^\infty}
+
L\int_0^t\XX(s)\di s
+
L\int_0^t\VV(s)\di s
\end{split}
\end{equation*}
Similarly arguing as in~\eqref{eq:stimo_V-V}, we also get that 
\begin{equation*}
\begin{split}
|V_1(t,p)-V_2(t,q)|
\le
|v-w|
+
\int_0^t
|E_1(s,X_1(s,p))-E_2(s,X_2(s,q))|
\di s
\end{split}
\end{equation*} 
for all $t\in[0,T]$, so that 
\begin{equation}
\label{eq:VV_EE}
\begin{split}
\VV(t)
&\le 
\int_{\R^{2d}\times\R^{2d}}|v-w|\di\pi_0(p,q)
\\
&\quad+
\int_0^t
\int_{\R^{2d}\times\R^{2d}}
|E_1(s,X_1(s,p))-E_2(s,X_2(s,q))|
\di\pi_0(p,q)
\di s
\\
&\le
\wass_1(\mu_0^1,\mu_0^2)
+
\int_0^t
\int_{\R^{2d}\times\R^{2d}}
|E_1(s,X_1(s,p))-E_2(s,X_2(s,q))|
\di\pi_0(p,q)
\di s
\end{split}
\end{equation}
for all $t\in[0,T]$ and so, in particular,
\begin{equation*}
\begin{split}
\XX(t)
&\le 
(1+Lt)\,\wass_1(\mu_0^1,\mu_0^2)
+
t\|F_1-F_2\|_{L^\infty}
+
L\int_0^t\XX(s)\di s
\\
&\quad
+
L\int_0^t
\int_0^s
\int_{\R^{2d}\times\R^{2d}}
|E_1(r,X_1(r,p))-E_2(r,X_2(r,q))|
\di\pi_0(p,q)
\di r
\di s
\end{split}
\end{equation*}
for all $t\in[0,T]$.
Now we have
\begin{equation*}
\begin{split}
|E_1(r,X_1(r,p))-E_2(r,X_2(r,q))|
&\le 
|E_1(r,X_1(r,p))-E_1(r,X_2(r,q))|
\\
&\quad+
|E_1(r,X_2(r,q))-E_2(r,X_2(r,q))|.
\end{split}
\end{equation*}
On the one side, since $f_1$ is a weak $\phi$-solution of~\eqref{eq:vp_gen} with respect to $(F_1,E_1)$, by~\eqref{eq:K_rho_cont} $E_1$ satisfies~\eqref{eq:E_cont}, and thus we can estimate
\begin{equation*}
|E_1(r,X_1(r,p))-E_1(r,X_2(r,q))|
\le 
\phi(|X_1(r,p)-X_2(r,q)|).
\end{equation*}
On the other side, again since $f_1$ and $f_2$ are weak $\phi$-solutions of~\eqref{eq:vp_gen}, we can write 
\begin{equation*}
\begin{split}
|E_1(r,&X_2(r,q))-E_2(r,X_2(r,q))|
\\
&=
\abs*{
\int_{\R^d}K(X_2(r,q),z)\,\rho_1(r,z)\di z
-
\int_{\R^d}K(X_2(r,q),z')\,\rho_2(r,z')\di z'
}
\\
&=
\abs*{
\int_{\R^{2d}}K(X_2(r,q),z)\,f_1(r,z,u)\di z\di u
-
\int_{\R^d}K(X_2(r,q),z')\,f_2(r,z',u')\di z'\di u'
}
\\
&=
\abs*{
\int_{\R^{2d}}K(X_2(r,q),X_1(r,o))\,f_0^1(o)\di o
-
\int_{\R^d}K(X_2(r,q),X_2(r,o'))\,f_0^2(o')\di o'
}
\end{split}
\end{equation*}
where in the last equality we changed variables, in virtue of~\eqref{eq:def_lagrangian}, letting $o=(z,u)$ and $o'=(z',u')$ for brevity. 
Since $\pi_0\in\plan(\mu_0^1,\mu_0^2)$, we can thus write
\begin{equation*}
\begin{split}
\bigg|
\int_{\R^{2d}}K(X_2(r,q)&,X_1(r,o))\,f_0^1(o)\di o
-
\int_{\R^d}K(X_2(r,q),X_2(r,o))\,f_0^2(o')\di o'
\bigg|
\\
&=
\bigg|
\int_{\R^{2d}}K(X_2(r,q),X_1(r,o))\di\mu_0^1(o)
-
\int_{\R^d}K(X_2(r,q),X_2(r,o))\di\mu_0^2(o')
\bigg|
\\
&=
\bigg|
\int_{\R^{2d}\times\R^{2d}}\big(
K(X_2(r,q),X_1(r,o))
-
K(X_2(r,q),X_2(r,o'))
\big)
\di\pi_0(o,o')
\bigg|
\\
&\le 
\int_{\R^{2d}\times\R^{2d}}\big|
K(X_2(r,q),X_1(r,o))
-
K(X_2(r,q),X_2(r,o'))
\big|
\di\pi_0(o,o')
\end{split}
\end{equation*}
Therefore, again changing variables in virtue of~\eqref{eq:def_lagrangian}, we get
\begin{equation*}
\begin{split}
&\int_{\R^{2d}\times\R^{2d}}
|E_1(r,X_2(r,q))-E_2(r,X_2(r,q))|
\di\pi_0(p,q)
\\
&\quad\le 
\int_{\R^{2d}\times\R^{2d}}
\int_{\R^{2d}\times\R^{2d}}\big|
K(X_2(r,q),X_1(r,o))
-
K(X_2(r,q),X_2(r,o'))
\big|
\di\pi_0(p,q)
\di\pi_0(o,o')
\\
&\quad=
\int_{\R^{2d}\times\R^{2d}}
\int_{\R^{2d}}\big|
K(h,X_1(r,o))
-
K(h,X_2(r,o'))
\big|
\,\rho_2(t,h)\di h
\di\pi_0(o,o')
\\
&\quad\le
\int_{\R^{2d}\times\R^{2d}}
\phi(|X_1(r,o)-X_2(r,o')|)
\di\pi_0(o,o').
\end{split}
\end{equation*}
Recalling that $\phi$ is concave, by Jensen's inequality we conclude that 
\begin{equation*}
\begin{split}
\int_{\R^{2d}\times\R^{2d}}
|E_1(r&,X_1(r,p))-E_2(r,X_2(r,q))|
\di\pi_0(p,q)
\\
&\le
2
\int_{\R^{2d}\times\R^{2d}}
\phi(|X_1(r,p)-X_2(r,q)|)
\di\pi_0(p,q)
\le2\,\phi(\XX(r)), 
\end{split}
\end{equation*}
so that 
\begin{equation}
\label{eq:XX_u}
\begin{split}
\XX(t)
&\le 
(1+Lt)\,\wass_1(\mu_0^1,\mu_0^2)
+
t\|F_1-F_2\|_{L^\infty}
+
L\int_0^t\XX(s)\di s
+
2L\int_0^t
\int_0^s
\phi(\XX(r))\di r
\di s
\end{split}
\end{equation}
for all $t\in[0,T]$.
In addition, recalling~\eqref{eq:VV_EE}, we also get that 
\begin{equation}
\label{eq:VV_phi}
\VV(t)
\le 
\wass_1(\mu_0^1,\mu_0^2)
+
2\int_0^t\phi(\XX(s))\di s
\end{equation} 
for all $t\in[0,T]$.
Now, letting $u\in W^{2,\infty}([0,T])$ be the function on the right-hand side of~\eqref{eq:XX_u}, we immediately get that $u,u'\ge0$ with $u(0)=
\wass_1(\mu_0^1,\mu_0^2)$ and
\begin{equation}
\label{eq:XX_u'}
\begin{split}
u'(t)
=
L\,\wass_1(\mu_0^1,\mu_0^2)
+
\|F_1-F_2\|_{L^\infty}
+
L\XX(t) 
+
2L
\int_0^t
\phi(\XX(s))\di s
\end{split}
\end{equation}
for all $t\in[0,T]$, so that,  
$u'(0)\le2L\,\wass_1(\mu_0^1,\mu_0^2)+\|F_1-F_2\|_{L^\infty}$.
Furthermore, we have
\begin{equation*}
\begin{split}
u''(t)
=
L\dot\XX(t) 
+
2L
\phi(\XX(t))
\end{split}
\end{equation*}
for a.e.\ $t\in(0,T)$. 
Note that, in virtue of the definition in~\eqref{eq:def_XX} and of problem~\eqref{eq:cauchy_problem_XV}, 
\begin{equation*}
\begin{split}
\dot\XX(t)
&\le
\int_{\R^{2d}\times\R^{2d}}
|\dot X_1(t,p)-\dot X_2(t,q)|\di\pi_0(p,q)
\\
&=
\int_{\R^{2d}\times\R^{2d}}
|F_1(t,X_1(t,p),V_1(t,p))-F_2(t,X_2(t,q),V_2(t,q))|\di\pi_0(p,q)
\\
&\le 
\norm*{F_1-F_2}_{L^\infty},
\end{split}
\end{equation*}
so that, recalling~\eqref{eq:XX_u} and~\eqref{eq:XX_u'} and since $\phi$ is non-decreasing, 
\begin{equation*}
u''(t)
\le 
L\norm*{F_1-F_2}_{L^\infty}
+
2L
\phi(\XX(t))
\le 
Lu'(t)
+2L\phi(u(t))
\end{equation*}
for a.e.\ $t\in(0,T)$.
Thanks to \cref{res:gronwall-2}, we thus conclude that, if 
\begin{equation*}
2L\,\wass_1(\mu_0^1,\mu_0^2)+\|F_1-F_2\|_{L^\infty}
<\delta
\end{equation*}
for some $\delta>0$, then 
\begin{equation*}
\XX(t)
\le 
\Psi_{\delta,2L}^{-1}
\big(
\Psi_{\delta,2L}(\wass_1(\mu_0^1,\mu_0^2))+e^{Lt}-1
\big)
\end{equation*}
for all $t\in[0,T]$. 
Moreover, from~\eqref{eq:VV_phi} and~\eqref{eq:XX_u'}, we also get that 
$\VV(t)\le u'(t)$, so that 
\begin{equation*}
\VV(t)
\le 
e^{Lt}
\big(
\delta+\sqrt{4L\Phi(\XX(t))}
\big)
\le 
e^{Lt}
\left(
\delta+\sqrt{4L\Phi\big(\Psi_{\delta,2L}^{-1}
\big(
\Psi_{\delta,2L}(\wass_1(\mu_0^1,\mu_0^2))+e^{Lt}-1
\big)
\big)}
\right)
\end{equation*}
for all $t\in[0,T]$, in virtue of \cref{res:gronwall-2}.
To conclude, we simply note that, by~\eqref{eq:plan_t},
\begin{equation*}
\begin{split}
\wass_1(\mu_1(t,\cdot),\mu_2(t,\cdot))
&\le
\int_{\R^{2d}\times\R^{2d}}
|p-q|\di\pi(t,p,q)
\\
&=
\int_{\R^{2d}\times\R^{2d}}
|\Gamma_1(t,p)-\Gamma_2(t,q)|\di\pi_0(p,q)
\le
\XX(t)+\VV(t)
\end{split}
\end{equation*}
for all $t\in[0,T]$, readily ending the proof.
\end{proof}

\section{Proofs of the main results}
\label{sec:proofs}

\subsection{Proof of \texorpdfstring{\cref{res:relation_class_modulus}}{mapping properties of K}}

We begin with the proof \cref{res:relation_class_modulus}.
Actually, we achieve the following slightly stronger result.
Here and in the following, the kernel~$K$ is as in~\eqref{eq:kernel}.

\begin{proposition}[Mapping properties of $K$] \label{res:mapping_kernel}
There is a dimensional constant $C_d>0$ with the following property.
If $\rho \in L^1(\R^d) \cap Y^\Theta_\ul(\R^d)$, then $K*\rho\in C^{0,\phi_\Theta}_b(\R^d)$, with
\begin{align}
\label{eq:bounded_est_kernel}
\|K * \rho\|_{L^\infty}
&\leq 
C_d 
\left(\norm{\rho}_{L^1}+\norm{\rho}_{Y^\Theta_\ul}\right),
\\
\label{eq:cont_est_kernel}
\int_{\R^d} \abs{ K(x-z) - K(y-z)}\,{\rho(z)}  \di z &\leq C_d \left( \norm{\rho}_{L^1} + \norm{\rho}_{Y^\Theta_\ul} \right) \phi_{\Theta}(\abs{x-y})
\quad
\forall x,y\in\R^d.
\end{align}
\end{proposition}

To prove \cref{res:mapping_kernel}, we need the following simple estimate, which generalizes~\cite{CS21}*{Eq.~(2.2)} to any dimension $d\ge2$.

\begin{lemma}[Oscillation]
\label{res:osc_kernel}
There exists a dimensional constant $C_d>0$ such that
\begin{equation} \label{eq:osc_gamma}
\abs{K(x-z) - K(y-z)} \leq C_d \left( \frac{1}{\abs{x-z}\abs{y-z}^{d-1}} + \frac{1}{\abs{y-z}\abs{x-z}^{d-1}} \right) \abs{x-y}
\end{equation}
for all $x,y,z\in\R^d$ with $x,y\ne z$.
\end{lemma}

\begin{proof}
We can assume $z=0$ without loss of generality. 
For $x,y\in\R^d\setminus\set*{0}$, we have
\begin{equation*}
\bigg| \frac{x}{\abs{x}^d} - \frac{y}{\abs{y}^d}\bigg|^2 
= 
\frac{1}{\abs{x}^{2(d -1)}} + \frac{1}{\abs{y}^{2(d -1)}} - \frac{ 2(x \cdot y)}{\abs{x}^d \abs{y}^d}
= 
\left[ \frac{\abs{x\abs{x}^{d -2}-y \abs{y}^{d -2} }}{\abs{x}^{d -1} \abs{y}^{d -1}}\right]^2, 
\end{equation*}
so that 
\begin{equation*}
\bigg|\frac{x}{\abs{x}^d } - \frac{y}{\abs{y}^d }\bigg| = \frac{\abs{x\abs{x}^{d -2}-y \abs{y}^{d -2} }}{\abs{x}^{d -1} \abs{y}^{d -1}}
\end{equation*} 
for all $x,y\in\R^d\setminus\set*{0}$.
Letting $F_d (\xi)= \xi \abs{\xi}^{d -2}$ for all $\xi \in \R^d$, we have 
$\abs{\nabla F_d (\xi)} \leq C_d  \abs{\xi}^{d -2}$ for all $\xi\in\R^d$, where $C_d >0$ is a dimensional constant.
Hence
\begin{equation*}
\abs*{x \abs{x}^{d-2} - y \abs{y}^{d-2} } 
\leq  
\abs{x-y}
\sup_{t \in [0,1]} \abs{\nabla F_d  (x+ t(x-y))}  
\leq 
C_d \,\abs{x-y}\sup_{t \in [0,1]} \abs{x+ t(x-y)}^{d -2}
\end{equation*}
for all $x,y\in\R^d$.
Since $d  \geq 2$, and thus the function $\xi \mapsto \abs{\xi}^{d -2}$ is convex, we can estimate 
\begin{equation*}
\abs{x + t(x-y)}^{d -2}
\leq 
(1-t)\abs{x}^{d  -2} + t \abs{y}^{d -2} 
\leq 
\abs{x}^{d -2} + \abs{y}^{d -2}
\end{equation*}
for all $x,y\in\R^d$.
Therefore, we get that  
\begin{equation*}
\bigg|\frac{x}{\abs{x}^d } - \frac{y}{\abs{y}^d }\bigg| 
= 
\frac{\abs{x\abs{x}^{d -2}-y \abs{y}^{d -2} }}{\abs{x}^{2(d -1)} \abs{y}^{d -1}} 
\leq 
C_d \, 
\abs{x-y} \left[ \frac{\abs{x}^{d -2} + \abs{y}^{d -2}}{\abs{x}^{d -1} \abs{y}^{d -1}} \right]
\end{equation*}
for all $x,y \in \R^d\setminus\set*{0}$,
yielding~\eqref{eq:osc_gamma} for $z=0$. 
\end{proof}

We can now prove \cref{res:mapping_kernel}. 
We follow the strategy of the proofs of~\cite{CS21}*{Th.~2.2 and Cor.~2.4}.
We also refer to the proofs of~\cite{HM18}*{Lem.~2.1} and~\cite{GS15}*{Ths.~A and~B}.

\begin{proof}[Proof of \cref{res:mapping_kernel}]
We write  
$K = K^1 + K^\infty$, with $K^1=K \mathbf{1}_{B_1}\in L^{\frac{d+1}d}(\R^d)$ and $K^\infty = K \mathbf{1}_{B_1^c}\in L^\infty(\R^d)$.
Since $\rho \in L^1 \cap L^{d+1}_\ul(\R^d)$, we can estimate
\begin{align*}
\abs{K} &* { \rho}(x) 
\leq 
\abs{K^1} * { \rho}(x) + \abs{K^\infty} * { \rho}(x)  
\leq 
\norm{K^1}_{L^{\frac{d+1}d}} 
\norm{\rho}_{L^{d+1}(B_1(x))} 
+ 
\norm{K^\infty}_{L^\infty} 
\norm{\rho}_{L^1} 
\\
& \leq 
\max\set*{\norm{K^1}_{L^{\frac{d+1}d}},\norm{K^\infty}_{L^\infty} } \left(\norm{\rho}_{L^{d+1}_\ul} + \norm{\rho}_{L^1}\right)
\le 
C_d \,(\norm{\rho}_{L^{d+1}_\ul} + \norm{\rho}_{L^1}) 
\\ 
& \leq 
C_d \left(\Theta(d+1) \,\norm{\rho}_{Y^\Theta_\ul} + \norm{\rho}_{L^1}\right) 
\leq 
C_d\, (\norm{\rho}_{Y^\Theta_\ul} + \norm{\rho}_{L^1})
\end{align*}
for all $x\in\R^d$, yielding~\eqref{eq:bounded_est_kernel}.
To  prove~\eqref{eq:cont_est_kernel}, fix $x,y \in \R^d$ and set $\e= \abs{x-y}$.
Due to~\eqref{eq:bounded_est_kernel}, we can assume $\e<e^{-d-1}$ without loss of generality. 
We write 
\begin{align*}
\int_{\R^d} |K(x-z)& - K(y-z)|\,{\rho(z)} \di z
\\ 
& =
\left( \int_{B_2(x)^c} + \int_{B_2(x) \setminus B_{2\e}(x)} + \int_{B_{2\e}(x)} \right) \abs{K(x-z) - K(y-z)}\,{\rho(z)} \di z.
\end{align*}
By \cref{res:osc_kernel}, we can estimate the first integral as 
\begin{align*}
\int_{B_2(x)^c} & \abs{K(x-z) - K(y-z)}\,{\rho(z)} \di z 
\\ 
& \leq 
C_d \,\abs{x-y} \int_{B_2(x)^c} \left( \frac{1}{\abs{x-z}\abs{y-z}^{d-1}} + \frac{1}{\abs{y-z}\abs{x-z}^{d-1}} \right)\,{\rho(z)} \di z
\\ 
& \leq 
C_d \, \abs{x-y} \,\norm{\rho}_{L^1}. 
\end{align*}
Concerning the second integral, since 
\begin{equation*}
\abs{y-z} \geq \frac{1}{2} \abs{x-z}
\quad
\text{for all}\ 
z \in B_2(x) \setminus B_{2\e}(x),
\end{equation*}
again by \cref{res:osc_kernel} we can estimate 
\begin{align*}
&\int_{B_2(x) \setminus B_{2\e}(x)}  \abs{K(x-z) -K(y-z)}\,{\rho(z)} \di z
\\ & \ \leq C_d\,\abs{x-y} \int_{B_2(x) \setminus B_{2\e}(x)} \left( \frac{1}{\abs{x-z}\abs{y-z}^{d-1}} + \frac{1}{\abs{y-z}\abs{x-z}^{d-1}} \right)\,{\rho(z)}\di z 
\\ &  \ \leq C_d\,\abs{x-y} \int_{B_2(x) \setminus B_{2\e}(x)} \frac{{\rho(z)}}{\abs{x-z}^d} \di z 
\leq C_{d}\,\abs{x-y} \,\norm{\rho}_{L^p(B_2(x))} \left( \int_{2\e}^2 r^{-d p' + d-1}  \di r \right)^{\frac{1}{p'}}
\\ & \  \leq C_{d}\,\abs{x-y} \,\norm{\rho}_{L^p_\ul} \left( \tfrac{2^{-d p'+d} (1 - \e^{-d p' +d})}{-d p' +d} \right)^{\frac{1}{p'}}
\leq C_{d}\,\abs{x-y} \,\norm{\rho}_{L^p_\ul}\, 2^{-\frac{d}{p}} (\e^{-\frac{d}{p-1}}-1)^{\frac{p-1}{p}} \left( \tfrac{p-1}{d}\right)^{\frac{p-1}{p}}
\\ & \
\leq C_{d}\,\abs{x-y}\, \norm{\rho}_{L^p_\ul}\, p\, \e^{-\frac{d}{p}} 
 \leq C_{d}\,p\, \Theta(p)\, \norm{\rho}_{Y^\Theta_\ul} \,\abs{x-y}^{1-\frac{d}{p}}. 
\end{align*}
for any $p>d+1$, with $p'$ the conjugate of $p$.
Finally, regarding the third and last integral, since $B_{2\e}(x) \subset B_{3\e}(y)$, we can estimate
\begin{align*}
\int_{B_{2\e}(x)} & \abs{K(x-z) - K(y-z)}\,{\rho(z)} \di z 
\leq \int_{B_{2\e}(x)} \frac{{\rho(z)}}{\abs{x-z}^{d-1}}  \di z + \int_{B_{3\e}(z)} \frac{{\rho(z)}}{\abs{y-z}^{d-1}} \di z 
\\ & \leq C_{d}\,\norm{\rho}_{L^p_\ul} \left( \int_0^{3\e} r^{(-d+1)p' +d-1}  \di r \right)^{\frac{1}{p'}}
\leq C_{d}\,\norm{\rho}_{L^p_\ul} \left( \tfrac{(3\e)^{(-d+1)p' + d}}{(-d+1)p'+d} \right)^{\frac{1}{p'}}
\\ & \leq C_{d}\,\norm{\rho}_{L^p_\ul} (3\e)^{1-\frac{d}{p}} \left( \tfrac{p-1}{p-d} \right)^{\frac{p-1}{p}}
\leq C_d\,p\,\Theta(p)\, \norm{\rho}_{Y^\Theta_\ul}\, \abs{x-y}^{1-\frac{d}{p}} 
\end{align*}
again for $p>d+1$.
Putting everything altogether, we conclude that  
\begin{align*}
\int_{\R^d} \abs{K(x-z) - K(y-z)}\,{\rho(z)} \di z
\leq 
C_d\left( \norm{\rho}_{L^1(\R^d)} + \norm{\rho}_{Y^\Theta_\ul} \right)\,p\, \Theta(p)\, \abs{x-y}^{1-\frac{d}{p}}
\end{align*}
for $x,y \in \R^d$ with $\abs{x-y}< e^{-d-1}$ and $p>d+1$.
In particular, choosing $p=-\log\abs{x-y}$, since $r^{\frac{d}{\log(r)}}=e^d$  for $r \in (0,1)$, we obtain that 
\begin{align*}
\int_{\R^d} |K(x-z) &- K(y-z)|\,{\rho(z)} \di z
\\
& \leq C_d \,(\norm{\rho}_{L^1} + \norm{\rho}_{Y^\Theta_\ul})\,|x-y| \,|\log\abs{x-y}|\, \Theta(|\log\abs{x-y}|)\, \abs{x-y}^{\frac{d}{\log{\abs{x-y}}}} 
\\
&\le 
C_d \,(\norm{\rho}_{L^1} + \norm{\rho}_{Y^\Theta_\ul})\, \phi_\Theta (\abs{x-y})
\end{align*}
for $x,y \in \R^d$ with $\abs{x-y}< e^{-d-1}$, completing the proof of~\eqref{eq:cont_est_kernel}.
\end{proof}

\subsection{Proof of \texorpdfstring{\cref{res:main_lagrangian}}{Lagrangian stability}}

In view of \cref{res:lagrangian_stab_vp_gen}, we just have to check that, if $f\in \class^\Theta([0,T])$ is a Lagrangian weak solution of~\eqref{eq:vp} in the sense of \cref{def:admissible_lagrangian}, then $f$ is a Lagrangian weak $\phi_\Theta$-solution of~\eqref{eq:vp_gen} with $F(t,x,v)=v$ for $t\in[0,T]$ and $x,v\in\R^d$ and $E_f=K*\rho_f$, where $K$ is as in~\eqref{eq:kernel}.
Indeed, we just need to check the validity of~\eqref{eq:K_rho_bounded} and~\eqref{eq:K_rho_cont}, but these respectively follow from~\eqref{eq:bounded_est_kernel} and~\eqref{eq:cont_est_kernel} in \cref{res:mapping_kernel}.
\qed

\begin{remark}[Relativistic case]
Note that the above argument \textit{verbatim} applies to the relativistic setting, that is, choosing $F(t,x,v)=\frac v{\sqrt{1+|v|^2}}$ for $t\in[0,T]$ and $x,v\in\R^d$. 
\end{remark}

\subsection{Proof of \texorpdfstring{\cref{res:main_existence}}{Existence}}

From now on, we assume $d\in\set*{2,3}$.
We begin with the following result, providing a suitable initial datum for the construction of the weak solution in \cref{res:main_existence}. 

\begin{lemma}[Datum]
\label{res:initial_datum}
If $\theta\colon\R^d\to\R$ satisfies~\eqref{eq:theta_existence}, then $f_0\colon\R^{2d}\to[0,+\infty)$ given by
\begin{equation}
\label{eq:initial_datum_theta}
f_0(x,v)=\frac{\mathbf 1_{(-\infty,0]}\left(|v|^2-\theta(x)^{\frac2d}\right)}{|B_1|\,\|\theta\|_{L^1}},
\quad
\text{for}\
x,v\in\R^d,
\end{equation}
satisfies $f_0\in L^1(\R^{2d})\cap L^\infty(\R^{2d})$, $f_0\,\mathscr L^{2d}\in\prob_1(\R^{2d})$ and, for some constant $C>0$,
\begin{equation}
\label{eq:moment_est_LP_datum}
\int_{\R^{2d}} \abs{v}^p\, f_0(x,v)  \di x  \di v 
\leq 
\frac{\|\theta\|_{L^{\frac pd+1}}^{\frac pd+1}}{\|\theta\|_{L^1}}
\quad
\text{for all}\ p\in[1,+\infty).
\end{equation} 
\end{lemma}

\begin{proof}
Note that $|v|\le\theta(x)^\frac1d$ for all $(x,v)\in\mathrm{supp} f_0$.
We thus have
\begin{equation}
\label{eq:rho_theta}
\rho_0(x)
=
\int_{\R^d} f_0(x,v)\di v
=
\frac{\mathscr L^d\big(\set*{v\in\R^d:|v|\le \theta(x)^\frac1d}\big)
}{|B_1|\|\theta\|_{L^1}}
=
\frac{\theta(x)}{\|\theta\|_{L^1}}
\end{equation}
for all $x\in\R^d$.
Consequently,
we can estimate
\begin{align*}
\int_{\R^{2d}} \abs{v}^p\, f_0(x,v)  \di x  \di v 
\le 
\int_{\R^{2d}} \abs{\theta(x)}^\frac pd\, f_0(x,v)  \di x  \di v
= 
\int_{\R^{2d}} \abs{\theta(x)}^\frac pd\, \rho_0(x)  \di x
=
\frac{\|\theta\|_{L^{\frac pd+1}}^{\frac pd+1}}{\|\theta\|_{L^1}},
\end{align*}
readily yielding the conclusion.
\end{proof}

We can now prove \cref{res:main_existence}.
Actually, we prove the following more precise result.

\begin{proposition}[Existence]
\label{res:lions-perthame} 
Assume that $\theta\in Y^\Theta(\R^d)$ satisfies~\eqref{eq:theta_existence}.
There exists a Lagrangian weak solution 
\begin{equation*}
f\in C([0,T];L^p(\R^{2d}))
\cap
L^\infty([0,T]\times\R^{2d})
\cap\class^\Theta([0,T])
\quad
\text{for all}\ 
p\in[1,+\infty)
\end{equation*}
of the system~\eqref{eq:vp} starting from~$f_0$ in~\eqref{eq:initial_datum_theta} of \cref{res:initial_datum} such that $f(t,\cdot)\,\mathscr L^{2d}\in\prob_1(\R^{2d})$, 
\begin{equation}
\label{eq:rho_C_LP}
\rho_f\in C([0,T];L^p(\R^d))
\quad
\text{for all}\
p\in[1,+\infty)
\end{equation}
and, for some constant $C_T>0$ depending on~$T$,
\begin{equation}
\label{eq:rho_est_LP}
\frac{\|\theta\|_{L^q}}{\|\theta\|_{L^1}}\le\|\rho_f\|_{L^\infty([0,T];L^q)}
\le C_T\|\theta\|_{L^q}
\quad
\text{for all}\ q\in\left[1,+\infty\right).
\end{equation}
\end{proposition}

\begin{proof}
By~\cite{LP91}*{Th.~1} (for $d=3$, the case $d=2$ being similar, see~\cites{M16,HM18}), there exists 
\begin{equation*}
f\in C([0,+\infty);L^p(\R^{2d}))
\cap
L^\infty([0,+\infty)\times\R^{2d})
\quad
\text{for all}\ 
p\in[1,+\infty)
\end{equation*}
a weak solution of the system~\eqref{eq:vp} starting from $f_0$ in~\eqref{eq:initial_datum_theta} of \cref{res:initial_datum} and such that 
\begin{equation}
\label{eq:moment_est_LP}
\sup_{t\in[0,T]}
\int_{\R^{2d}} \abs{v}^p\, f(t,x,v)  \di x  \di v 
<+\infty
\quad
\text{for all}\ p\in[1,+\infty).
\end{equation}
Note that the notion of weak solution here is well-posed in the sense of \cref{def:weak_sol_vp}, since $E_f\in L^\infty([0,T]\times\R^d)$ in virtue of~\eqref{eq:moment_est_LP} and~\cite{LP91}*{Eq.~(16)}.
Moreover, $f$ is constant along characteristic curves of~\eqref{eq:vp_ODE} which are defined almost everywhere.
Finally, by~\cite{LP91}*{Eq.~(8)} and~\eqref{eq:moment_est_LP_datum}, we get~\eqref{eq:rho_C_LP}. 
Thus, we just need to show~\eqref{eq:rho_est_LP}, so that $f\in\class^\Theta([0,T])$ in particular.
For the first inequality in~\eqref{eq:rho_est_LP}, we observe that 
\begin{equation*}
\|\rho_f\|_{L^\infty(L^q)}
\ge 
\|\rho_f(0,\cdot)\|_{L^q}
=
\|\rho_0\|_{L^q}
=
\frac{\|\theta\|_{L^q}}{\|\theta\|_{L^1}}
\end{equation*}
because of~\eqref{eq:rho_theta} and~\eqref{eq:rho_C_LP}.
For the second inequality in~\eqref{eq:rho_est_LP}, we argue as in~\cite{M16}*{Sec.~3}. 
By~\cite{LP91}*{Eq.~(14)}, we can estimate \begin{equation*}
\|\rho_f(t,\cdot)\|_{L^{\frac pd +1}}
\le 
CM_p(t)^{\frac d{p+d}}
\quad
\text{for}\ t\in[0,T],
\end{equation*}
for some constant $C_T>0$ independent of $p$ and $t\in[0,T]$, but dependent on $T>0$, which may vary from line to line in what follows, where 
\begin{equation*}
M_p(t)
=
\int_{\R^{2d}} \abs{v}^p\, f(t,x,v)  \di x  \di v.
\end{equation*}
Exploiting~\eqref{eq:vp_ODE} and the fact that $f$ is constant along characteristics, we can estimate
\begin{equation*}
M_p(t)\le M_p(0)+C_T\,p\int_0^tM_p(s)^{1-\frac1p}\di s
\end{equation*} 
By a simple Gr\"onwall-type argument, we infer that 
\begin{equation*}
\sup_{t\in[0,T]}M_p(t)\le M_p(0)+C_T^p
\quad
\text{for all}\ t\in[0,T].
\end{equation*}
Since $f(0,\cdot)=f_0$, by~\eqref{eq:moment_est_LP_datum} we get
\begin{equation*}
M_p(t)^{\frac d{p+d}}
\le
\left(\frac{\|\theta\|_{L^{\frac pd+1}}^{\frac pd+1}}{\|\theta\|_{L^1}}
+C_T^p\right)^{\frac d{p+d}}
\le C_T\,\|\theta\|_{L^{\frac pd+1}},
\end{equation*} 
proving the second inequality in~\eqref{eq:rho_est_LP} and ending the proof.
\end{proof}

\subsection{Proof of \texorpdfstring{\cref{res:saturation}}{Saturation}}

We need some notation and the preliminary \cref{res:iterated_log_growth} below.
For each $m\in\mathbb N$, we define $\ell_m\colon[0,+\infty)\to[0,+\infty)$ by letting 
\begin{equation}
\ell_{m}(r)  = 
\mathbf{1}_{(0, \e_{m})}(r)
\,\log_{m}(r) 
\quad
\text{for all}\ r\ge0,
\label{iterated log}
\end{equation}
where $\e_m\in(0,1)$ is such that $\log_{m}(\e_m)=-1$ (recall the notation in~\eqref{eq:logol}).

\begin{lemma}
\label{res:iterated_log_growth}
For $m\in\mathbb N$, there are $p_m\in[1,+\infty)$ and $0<a_m<b_m<+\infty$ such that
\begin{equation}
\label{eq:iterated_log_growth}
a_m\log_{m-1}(p)
\le 
\norm*{\ell_m(|\cdot|)}_{L^p}
\le 
b_m\log_{m-1}(p)
\quad
\text{for all}\ p\ge p_m.
\end{equation}  
\end{lemma}

\begin{proof}
Given $p\ge \log(1/\e_m)$, we can easily estimate
\begin{equation}
\label{eq:iterated_sotto}
\begin{split}
\norm*{\ell_m(|\cdot|)}_{L^p}^p
=
\int_{B_{\e_m}}|\log_m(|x|)|^p\di x
\ge 
\int_{B_{e^{-p}}}|\log_m(|x|)|^p\di x
\ge 
C_d\,e^{-dp}|\log_{m-1}(p)|^p
\end{split}
\end{equation}
for all $m\in\mathbb N$, proving the lower bound in~\eqref{eq:iterated_log_growth}. 
For the upper bound in~\eqref{eq:iterated_log_growth}, we argue by induction.
If $m=1$, then by direct computation we have
\begin{equation*}
\norm*{\ell_1(|\cdot|)}_{L^p}^p
=
\int_{B_1}|\log(|x|)|^p\di x
=
C_d\int_0^1(-\log r)^p\,r^{d-1}\di r
=
C_d\,d^{-(p+1)}\,\Gamma(p+1)
\end{equation*}
and the desired upper bound readily follows by Stirling's formula.
If $m\ge2$, then
\begin{equation*}
\begin{split}
\norm*{\ell_m(|\cdot|)}_{L^p}
&=
\left(
\int_{B_{\e_m}}|\log_m(|x|)|^p\di x
\right)^{1/p}
\\
&=
\frac{|B_{\e_m}|^{1/p}}{p}
\left(
\mint{-}_{B_{\e_m}}\Big|\log\big(
\log_{m-1}(|x|)
\big)^p\Big|^p\di x
\right)^{1/p}.
\end{split}
\end{equation*} 
Now $r\mapsto(\log r)^p$ is concave on $\left[e^{p-1},+\infty\right)$.
Since $\log_{m-1}(\e_m)=-e$, for $p\ge2$  we have
\begin{equation*}
\begin{split}
\mint{-}_{B_{\e_m}}\Big|\log\big(
\log_{m-1}(|x|)
\big)^p\Big|^p\di x
&\le 
\left(
\log\left(
\mint{-}_{B_{\e_m}}
\big|
\log_{m-1}(|x|)
\big|^p\di x
\right)
\right)^p
\\
&\le 
p^p
\Big(
\log\left(
|B_{\e_m}|^{-1/p}
\,
\norm*{\ell_{m-1}(|\cdot|)}_{L^p}
\right)
\Big)^p
\end{split}
\end{equation*}
by Jensen's inequality, so that
\begin{equation*}
\begin{split}
\norm*{\ell_m(|\cdot|)}_{L^p}
\le 
|B_{\e_m}|^{1/p}
\log\left(
|B_{\e_m}|^{-1/p}
\,
\norm*{\ell_{m-1}(|\cdot|)}_{L^p}
\right),
\end{split}
\end{equation*}
readily yielding the conclusion.
\end{proof}

\begin{proof}[Proof of \cref{res:saturation}]
For each $m\in\mathbb N$, there exists $\delta_m>0$ such that 
\begin{equation*}
\phi_{\Theta_m}(r)
=
r\,|\log r|\,\Theta_m(\abs*{\log r})
=
\Theta_{m+1}(r)
\quad
\text{for all}\ r\in[0,\delta_m].
\end{equation*}
Hence $\phi_{\Theta_m}$ is concave on $[0,\delta_m]$ with $\phi_{\Theta_m}(0)=0$.
Therefore, we can estimate 
\begin{equation*}
\Phi_{\Theta_m}(t)
=
\int_0^t\phi_{\Theta_m}(s)\di s
\le 
t\,\phi_{\Theta_m}(t)
=
t\,\Theta_{m+1}(t)
\quad
\text{for all}\
t\in[0,\delta_m].
\end{equation*}
In particular, we readily infer that
\begin{equation*}
\begin{split}
\lim_{\e\to0^+}
\int_{\e}
^{\delta_m}
\frac{\di t}{\sqrt{\Phi_{\Theta_m}(t)}}
&\ge 
\lim_{\e\to0^+}
\int_{\e}
^{\delta_m}
\frac{\di t}{\sqrt{t\,\Theta_{m+1}(t)}}
\\
&=
\lim_{\e\to0^+}
\int_{\e}
^{\delta_m}
\frac{\di t}{t\,\abs{\log t}\, \abs{\log_2(t)} \cdots \abs{\log_{m+1}(t)} 
}
=+\infty,
\end{split}
\end{equation*}
so that $\Phi_{\Theta_m}$ satisfies~\eqref{eq:Phi_Theta_osgood}.
To conclude, we define $\theta_m\colon\R^d\to[0,+\infty)$ as
\begin{equation*}
\theta_m(x)
=
\ell_1(|x|)\,\ell_2(|x|)^2\dots\ell_{m+1}(|x|)^2
\quad
\text{for}\ x\in\R^d.
\end{equation*}
On the one side, arguing as in~\eqref{eq:iterated_sotto}, we easily see that 
\begin{equation*}
\begin{split}
\norm*{\theta_m}_{L^p}^p
&\ge
\int_{B_{e^{-p}}}
|\log_1(|x|)|^{p}\,|\log_2(|x|)|^{2p}\dots|\log_{m+1}(|x|)|^{2p}\di x
\\
&\ge 
C_d\,e^{-dp}\,p^{p}\,|\log_1(p)|^{2p}\dots|\log_{m}(p)|^{2p}
=C_d\,e^{-dp}\,\Theta_{m}(p)^p
\end{split}
\end{equation*}
for all $p\in[1,+\infty)$.
On the other side, by \cref{res:iterated_log_growth} and H\"older's inequality, we get
\begin{equation*}
\begin{split}
\|\theta_m\|_{L^p}
&\le 
\|\ell_1(|\cdot|)\|_{L^{(m+1)p}}
\,
\|\ell_2(|\cdot|)^2\|_{L^{(m+1)p}}
\dots
\|\ell_{m+1}(|\cdot|)^2\|_{L^{(m+1)p}}
\\
&=
\|\ell_1(|\cdot|)\|_{L^{(m+1)p}}
\,
\|\ell_2(|\cdot|)\|_{L^{2(m+1)p}}^2
\dots
\|\ell_{m+1}(|\cdot|)\|_{L^{2(m+1)p}}^2
\\
&\le 
C_m\,p\,\log_1(p)^2\dots\log_m(p)^2
=
C_m\,\Theta_m(p)
\end{split}
\end{equation*}
for all $p\ge p_m$ for some constant $C_m>0$ depending on $m$ only, yielding the conclusion.
\end{proof}

\begin{remark}[Saturation of $\Theta_\alpha(p)=p^{1/\alpha}$]
Fix $\alpha\in[1,\infty)$.
Arguing as above, one can easily see that $\theta_\alpha(x)=\ell_1(|x|)^{1/\alpha}$, for $x\in\R^d$, saturates the growth function $\Theta_\alpha(p)=p^{1/\alpha}$ in the sense of \cref{res:saturation}, giving an alternative proof of~\cite{HM18}*{Prop.~1.14}. 
\end{remark}


\end{document}